\documentclass[10pt]{article}

\usepackage[a4paper, top=1in, bottom=1in, left=1in, right=1in]{geometry}
\usepackage{amssymb, amscd, amsmath, amsthm, amsfonts,color,wasysym,graphics, graphicx, xcolor, url, hyperref, hypcap, frcursive,xparse,comment}
\usepackage{multicol}
\usepackage{verbatim}
\usepackage{authblk}

\newcommand{\idiot}[1]{\vspace{5 mm}\par \noindent
\marginpar{\textsc{Note}}
\framebox{\begin{minipage}[c]{1.05 \textwidth}
#1 \end{minipage}}\vspace{5 mm}\par}

\newcommand{\todone}[1]{\vspace{5 mm}\par \noindent
\marginpar{\textsc{DONE!}}
\framebox{\begin{minipage}[c]{1.05 \textwidth}
\tt #1 \end{minipage}}\vspace{5 mm}\par}

\renewcommand{\todone}[1]{} 

\renewcommand{\idiot}[1]{}

\newdimen\squaresize \squaresize=12pt
\newdimen\thickness \thickness=0.4pt

\def\square#1{\hbox{\vrule width \thickness
    \vbox to \squaresize{\hrule height \thickness\vss
       \hbox to \squaresize{\hss#1\hss}
    \vss\hrule height\thickness}
\unskip\vrule width \thickness}
\kern-\thickness}

\def\vsquare#1{\vbox{\square{$#1$}}\kern-\thickness}

\def\young#1{
\vbox{\smallskip\offinterlineskip
\halign{&\vsquare{##}\cr #1}}}

\def\thisbox#1{\kern-.09ex\mathfrak{b}ox{#1}}
\def\downbox#1{\lower1.200em\hbox{#1}}

\definecolor{grey}{rgb}{0.8,0.8,0.8} % grey color for filled square
\newcommand{\grey}{\color{grey}} % darkblue command
\def\gsqr{\grey \vrule height\squaresize width\squaresize}%

\makeatletter
\newtheorem*{rep@theorem}{\rep@title}
\newcommand{\newreptheorem}[2]{%
\newenvironment{rep#1}[1]{%
 \def\rep@title{#2 \ref{##1}}%
 \begin{rep@theorem}}%
 {\end{rep@theorem}}}
\makeatother

\newtheorem{thm}{Theorem}[section]
\newreptheorem{thm}{Theorem}
\newtheorem*{thm*}{Theorem}
\newtheorem{prop}[thm]{Proposition}
\newreptheorem{prop}{Proposition}
\newtheorem*{prop*}{Proposition}
\newtheorem{lemma}[thm]{Lemma}
\newreptheorem{lemma}{Lemma}
\newtheorem*{lemma*}{Lemma}
\newtheorem{cor}[thm]{Corollary}
\newreptheorem{cor}{Corollary}
\newtheorem*{cor*}{Corollary}
\theoremstyle{definition}
\newtheorem{defn}[thm]{Definition}
\newreptheorem{defn}{Definition}
\newtheorem*{defn*}{Definition}
\newtheorem{eg}[thm]{Example}
\newtheorem*{eg*}{Example}

\newtheorem{remark}[thm]{Remark}
\newcommand{\K}{\{0,1,\dots,k\}}
\newcommand{\Pk}{\mathcal{P}^{(k)}}
\newcommand{\Ckp}{\mathcal{C}^{(k+1)}}

\newcommand{\tr}{z}

\def\bfh{{\bf h}}
\def\uu{{\bf u}}

\def\lbl{{L}}

\def\la{{\lambda}}

\def\core{{\lambda}}
\def\bndd{{\mu}}
\def\wt{{\Lambda}}
\def\Anull{{A_\emptyset}}
\def\hook{{\mathrm{hook}}}
\def\cella{{\mathrm{add}}}
\def\cellr{{\mathrm{del}}}
\def\inv{{\mathrm{inv}}}
\def\Acos{{\widetilde{\mathcal{A}}^0_k}}
\def\A{{\mathcal{A}}}
\def\Ct{{\tilde C}}
\def\Cb{{\bar C}}
\def\Ch{{\hat C}}

\definecolor{darkblue}{rgb}{0,0,0.7} % darkblue color
\newcommand{\darkblue}{\color{darkblue}} % darkblue command
\newcommand{\deft}[1]{\emph{\darkblue #1}} % emphasis of a definition

\begin{document}

\title{Symmetries on the lattice of $k$-bounded partitions.\thanks{This work is supported in part by CRC and NSERC.
This paper originated in a working session at the Algebraic
Combinatorics Seminar at the Fields Institute with the active
participation of C.~Benedetti, N.~Bergeron, Z.~Chen, H.~Heglin, D.~Mazur and H. Thomas.}
\thanks{This research was facilitated by computer exploration using the open-source
mathematical software \texttt{Sage}~\cite{sage} and its algebraic
combinatorics features developed by the \texttt{Sage-Combinat}
community~\cite{sage-co}.
 }}  

\author[4]{Chris Berg}
\author[2]{Nathan Williams}
\author[1,3] {Mike Zabrocki}
\affil[1]{Fields Institute\\ Toronto, ON, Canada}
\affil[2]{Universit\'e du Qu\'ebec \`a Montr\'eal, Montr\'eal, QC, Canada}
\affil[3]{York University\\ Toronto, ON, Canada}
\affil[4]{Google}

\maketitle

\begin{abstract}
In 2002 R.\ Suter \cite{su} identified a dihedral symmetry on certain order ideals in Young's lattice
and gave a combinatorial action on the partitions in these order ideals.
Viewing this result geometrically, the order ideals can be seen to be seen to be in bijection with
the alcoves in a $2$-fold dilation in the geometric realization of the affine symmetric group.  By considering
the $m$-fold dilation we observe a larger set of order ideals in
the $k$-bounded partition lattice that was considered by L.\ Lapointe, A.\ Lascoux, and J.\ Morse \cite{LLM}
in the study of $k$-Schur functions.  We identify the order ideal and the cyclic action on it
explicitly in a geometric and combinatorial form.
\end{abstract}

\noindent
%Running Title: Symmetries on the lattice of $k$-bounded partitions\\

%%%%%%%%%%%%%%%%%%%%%%%%%%%%%%%%%%%%%%%%%%%%%%%%%%%%%%%%%%%%%%%%%%%%%
\section{Introduction}
\label{sec:intro}
%%%%%%%%%%%%%%%%%%%%%%%%%%%%%%%%%%%%%%%%%%%%%%%%%%%%%%%%%%%%%%%%%%%%%

For each $k \in \mathbb{N}$, R.\ Suter described a set $Y^k$ of partitions with unexpected 
dihedral symmetries~\cite{su}.  The set $Y^k$ is an order ideal in Young's lattice, and so to 
specify $Y^k$ it is enough to specify that its maximal elements are exactly those rectangles with hook-lengths at most $k$.

\begin{repdefn}{defn:yk}[\cite{su}]
%\label{defn:yk}
Let $R_i := (i^{k+1-i})$ for $1 \leq i \leq k$.  
Then \[Y^k := \{ \core : \core \subseteq R_i \text{ for some } 1 \leq i \leq k\}.\]
\end{repdefn}

R.\ Suter proved that $Y^k$ has the same symmetries as the affine Dynkin diagram of type $\widetilde{\A}_k$ by explicitly describing a cyclic action on $\core \in Y^k$ which, along with conjugation of a partition, gives a dihedral action on $Y^k$.

Since any $\core \in Y^k$ has maximum hook-length at most $k$, $\core$ may be viewed as a 
$k$-bounded partition (or $(k+1)$-core), so that $Y^k$ can be equivalently described as an 
order ideal in the lattice of $k$-bounded partitions (defined in Section~\ref{sec:kbounded}).  
It is well-known that the dominant alcoves in type $\widetilde{\A}_k$ are indexed by $k$-bounded partitions.  

\begin{repthm}{thm:all_together}[\cite{LM2}, \cite{Lascoux}]%{thm:core_to_grassmannian}
There is an order-preserving bijection between the lattice of $k$-bounded partitions, 
the lattice of $(k+1)$-cores, and weak order on the dominant alcoves in type $\widetilde{\A}_k.$
\end{repthm}

We can therefore associate to each partition $\core \in Y^k$ a dominant alcove $A_\core$.  
More precisely, R.\ Suter showed in~\cite{su2} that $Y^k$ is in bijection with alcoves in 
$2\Anull$, the two-fold dilation of the fundamental alcove $\Anull$ in type $\widetilde{\A}_k$.  
As the fundamental alcove has a $(k+1)$-fold cyclic symmetry, so does $2\Anull$---and so does $Y^k$.  
Thus, the natural geometric symmetry on $2\Anull$ explains the unexpected symmetry of $Y^k$.

The bijection between $Y^k$ and $2\Anull$ sends the maximal elements $R_i \in Y^k$ to the alcoves 
$A_{R_i} \in 2\Anull$ with a facet on the hyperplane $H_{\alpha_0,2}=\{x:\langle \alpha_0,x \rangle =2\}$ 
(where $\alpha_0$ is the highest root) and a single vertex on the
hyperplane $H_{\alpha_0,1}=\{x:\langle \alpha_0,x \rangle =1\}$.  
Such an alcove is characterized by the unique fundamental weight 
$\wt_i$ that is a vertex of $A_{R_i}$ not on the hyperplane $H_{\alpha_0,2}$.  We emphasize this with 
correspondence with a proposition.

\begin{repprop}{prop:rectangles_and_fund_weights}
The map $R_i \mapsto \wt_{i}$ is a bijection between the maximal elements of $Y^k$ and the fundamental weights.%There is a bijection between the rectangles $R_i$ and
\end{repprop}

This first motivation is reviewed in more detail in Section~\ref{sec:suter_symmetry}.

\vspace{1em}

Our second motivation comes from $k$-Schur functions, which---like the dominant alcoves in 
type $\widetilde{\A}_k$---are also indexed by $(k+1)$-cores $\core$ or $k$-bounded partitions $\bndd$.  Proposition~\ref{prop:rectangles_and_fund_weights} has an algebraic analogue in the theory, whereby the $k$-Schur function $s_{R_i}$ may be expressed as a sum over the $\A_k$-orbit of the fundamental weight $\wt_i$.  These rectangles $R_i$ also
appear in L.\ Lapointe and J.\ Morse's paper~\cite{LM3}, where they prove the following theorem.%It is here that we draw on the established analogy to $k$-Schur functions.

\begin{repthm}{thm:LM}[Theorem 40~\cite{LM3}]
%\label{thm:LM}
For a rectangle $R_i$ and a $k$-bounded partition $\bndd$, 
\[s_\bndd^{(k)} s_{R_i}^{(k)} = s_{\bndd \cup R_i} ^{(k)},\] 
where $\bndd_1 \cup \bndd_2$ denotes the partition obtained by combining the parts of 
$\bndd_1$ and $\bndd_2$ and placing them into non-increasing order.
\end{repthm}

In Section \ref{sec:gen}, we consider products of the form
\[s^{(k)}_{\cup_{j=1}^{m-1} R_{i_j}} = \prod_{j=1}^{m-1} s^{(k)}_{R_{i_j}},\]
which give a generalization of the algebraic analogue of Proposition~\ref{prop:rectangles_and_fund_weights}, connecting partitions indexed by $\cup_{j=1}^{m-1} R_{i_j}$ and
sums of weights $\sum_{j=1}^{m-1} \Lambda_{i_j}$.%present algebraic %reasons for considering the

We are led to define a generalization of $Y^k$ 
to the $k$-bounded partition lattice of L.\ Lapointe, A.\ Lascoux and J.\ Morse~\cite{LLM} by taking 
all $k$-bounded partitions $\bndd$ such that the corresponding alcove $A_\bndd$ is contained in the $m$-fold dilation of the fundamental alcove $m\Anull$.  By construction, the geometric symmetry on $m\Anull$ induces a symmetry on $Y_m^k$, generalizing 
R.\ Suter's construction when $m=2$.

%By considering other integral dilations of the fundamental alcove, we define a generalization of $Y^k$ 
%to the $k$-bounded partition lattice of L.\ Lapointe, A.\ Lascoux and J.\ Morse~\cite{LLM} by taking 
%all $k$-bounded partitions $\bndd$ such that the corresponding alcove $A_\bndd$ is contained in $m\Anull$.  
%By construction, the geometric symmetry on $m\Anull$ induces a symmetry on $Y_m^k$, generalizing 
%R.\ Suter's construction when $m=2$.

\begin{repdefn}{def:ymk}
Let $Y_m^k := \{ \bndd : A_\bndd \in m\Anull\}.$
\end{repdefn}

Exactly as in the case of $Y^k$---since the bijection between $k$-bounded partitions and alcoves is order 
preserving---$Y_m^k$ is an order ideal in the lattice of $k$-bounded partitions and so it is enough to 
characterize the maximal elements.  As in the $m=2$ case, the maximal alcoves in the $m\Anull$ have a
facet which is on the $H_{\alpha_0,m} 
= \{ x: \left< \alpha_0, x \right> = m \}$ hyperplane and are in bijection 
with the dominant weights on the hyperplane $H_{\alpha_0,m-1} 
= \{ x: \left< \alpha_0, x \right> = m-1 \}$.  These weights are all sums of $m-1$ fundamental 
weights of the form $\sum_{j=1}^{m-1} \wt_{i_j}$.

Led by the connection to $k$-Schur functions and the bijection between weights $\sum_{j=1}^{m-1} \wt_{i_j}$ and maximal alcoves in 
$m\Anull$, we arrive at our main theorem: a characterization of the maximal elements of $Y_m^k$.

\begin{repthm}{thm:mainthm}
The maximal $k$-bounded partitions in $Y_m^k$ are the partitions of the form $\cup_{j=1}^{m-1} R_{i_j}.$
\end{repthm}

Section~\ref{sec:proofs} is devoted to proving Theorem~\ref{thm:mainthm}.  We show that
there are ${m+k-2}\choose{k-1}$ alcoves in $m \Anull$ with a face on the hyperplane $H_{\alpha_0,m}$.  These alcoves are indexed by the $k$-bounded partitions which are unions 
of $m-1$ maximal rectangles.

In Section~\ref{sec:applications}, 
we consider applications of Definition~\ref{def:ymk} and Theorem~\ref{thm:mainthm}.  We first show how to directly reconstruct a $(k+1)$-core or the $k$-bounded partition
given the weights corresponding to the vertices of an alcove $A_\core$.  In terms of $k$-bounded partitions,
the answer may be given as follows.
\begin{repcor}{cor:intersect}
Let $A_\bndd$ be an alcove with vertices $v_d = \sum_{j=1}^k \wt_{i_j}$ for $0 \leq d \leq k$ and
let $\bndd_d = \cup_{j=1}^{m-1} R_{i_j}.$  Then
$\bndd$ is the largest partition that is contained in of all of the $\bndd_d$.
\end{repcor}

We next give an explicit formula for a geometric map, generalizing R.\ Suter's symmetry on $Y^k = Y_2^k$ to $Y_m^k$.
Using the map from alcoves to $k$-bounded partitions and $(k+1)$-cores, we show how this
map is computed on $Y_m^k$.  In the case that $m$ and $k+1$ are relatively prime, there is a central alcove of $Y_m^k$ fixed by
rotation.  In Section~\ref{subsec:translate}, we describe a region that is similar to $Y_m^k,$ but where this central alcove has been sent to the fundamental alcove.
%identity alcove is at the center of the region.

We finally show that $Y^k$ under R.~Suter's cyclic action exhibits the cyclic sieving phenomenon, proving conjectures of V.~Reiner and D.~Stanton.  The general result for $Y_m^k$ was proven in~\cite{TW}.

%One thing that Suter does do is recognize certain alcoves in the dilated fundamental alcove as corresponding to abelian ideals for the corresponding Lie algebra.

%{\bf MZ}: I want to say something like: well we are not ignoring other
%types, but the combinatorics isn't as interesting or as easy to state
%as we observe on partitions in type A.
%}
%{\bf CB:} Not sure what it says about the symmetry in general. Suter points out 
%that the symmetry coming from the Dynkin diagram is not the same as the symmetry 
%of the Hasse graph in type $C_3$ and $G_2$...

%\todone{explain what Su2 says about non-type A

%%%%%%%%%%%%%%%%%%%%%%%%%%%%%%%%%%%%%%%%%%%%%%%%%%%%%%%%%%%%%%%%%%%%%%%%%%%%%%%%%%%%%%%%%%
\section{The combinatorics of $\widetilde{\A}_k$}
\label{sec:comb_affine_A}
%%%%%%%%%%%%%%%%%%%%%%%%%%%%%%%%%%%%%%%%%%%%%%%%%%%%%%%%%%%%%%%%%%%%%%%%%%%%%%%%%%%%%%%%%%

In this section we review the geometry of affine type $\widetilde{\A}_k$, 
Grassmannian permutations, $(k+1)$-cores, and $k$-bounded partitions to establish the
dictionary associated with the following well-known result.

\begin{thm}[\cite{LM2}, \cite{Lascoux}]
\label{thm:all_together}
There is an order-preserving bijection between the lattice of $(k+1)$-cores, the lattice of 
$k$-bounded partitions, and weak order on the dominant alcoves in type $\widetilde{\A}_k.$
\end{thm}

%%%%%%%%%%%%%%%%%%%%%%%%%%%%%%%%%%%%%%%%%%%%%%%%%%%%%%%%%%%%%%%%%%%%%%%%%%%%%%%%%%%%%%%%%%
\subsection{$\widetilde{\A}_k$ and Affine Grassmannian elements}
%%%%%%%%%%%%%%%%%%%%%%%%%%%%%%%%%%%%%%%%%%%%%%%%%%%%%%%%%%%%%%%%%%%%%%%%%%%%%%%%%%%%%%%%%%

Let $\Delta:=\{\alpha_i\}_{1 \leq i \leq k}$ be the set of \deft{simple roots} of type $\A_k$.  
These form the basis for a vector space $V$ with a symmetric bilinear form 
$\langle \cdot,\cdot \rangle$ given by:
 \[\langle \alpha_i, \alpha_j \rangle = \left\{
\begin{array}{ll}
2 & \textrm{if }i=j,\\
-1 & \textrm{if } i=j\pm  1,\\
0 & \textrm{otherwise.} 
\end{array}
\right.\]

For $v \in V$, we let $H_{v,p}:=\{x \in V : \langle v,x\rangle = p\}$ and $H_{v}:=H_{v,0}$.  %We write $H_i:=H_{\alpha_i}$ for $1 \leq i \leq k$.
%$ denote the hyperplane through the origin,perpendicular to $v$. 
Let $s_i$ be the reflection of a vector $v$ through the hyperplane $H_{\alpha_i}$
so that the set of reflections $\{s_i\}_{1 \leq i \leq k}$ are the reflections in the
hyperplanes perpendicular to the simple roots.  These elements generate a group we denote by $\A_k$, which is isomorphic to the symmetric group on $k+1$ letters. 

Let $\Phi:=\{ w \alpha_i : w \in \A_k, \alpha_i \in \Delta\}$ be the set of \deft{roots} of type $\A_k$.  
The \deft{affine arrangement} is the set of hyperplanes $\{ H_{\alpha,p} : \alpha\in\Phi, p\in \mathbb Z\}$.  
We write $\{\Lambda_i\}_{1 \leq i \leq k}$ for the set of \deft{fundamental weights}---the basis dual to the 
simple roots, defined by $\langle \alpha_i,\Lambda_{i} \rangle=1$.  
The $\mathbb{Z}$-span of the $\{ \Lambda_i \}_{1\leq i \leq k}$ are the \deft{weights}.  
These are also the zero dimensional intersections of the $\{H_{\alpha,p}\}_{\alpha \in \Phi}$.  
The element $\alpha_0 = \alpha_1+\dots+\alpha_k \in \Phi$ is the
\deft{highest root}.  From the definition, $\langle \alpha_0, \Lambda_i \rangle=1$ for $1\leq i \leq k$.

The \deft{dominant chamber} is the closed region bounded by the hyperplanes 
$H_{\alpha_i,0}$; it is also the nonnegative span of the fundamental weights.  We denote the dominant 
chamber by $C = \{ \sum_{i=1}^k b_i \Lambda_i : b_i \geq 0 \}$.  A weight is called \deft{dominant} 
if it lies in the dominant chamber.  The \deft{fundamental alcove} is the closed region bounded by the 
walls of the dominant chamber, together with the hyperplane $H_{\alpha_0,1}$.  We denote it by 
$\Anull:= \{ \sum_{i=1}^k b_i \Lambda_i :  b_i \geq 0 \hbox{ and } \sum_{i=1}^k b_i \leq 1 \}$.
%\{ (a_1, \dots, a_{k+1}) \in V: a_1\geq a_2\geq \dots \geq a_{k+1} \geq a_{1}-1\}$.
%%%% We don't have coordinates here so we need to give a different description.
%% My note: \Lambda_i = \epsilon_1 + ... + \epsilon_i = (1^i,0^{k+1-i})

The \deft{affine symmetric group} $\widetilde{\A}_k$ is generated by the simple reflections of 
type $\A_k$ along with an additional generator $s_0$, which
acts as reflection in $H_{\alpha_0,1}$.  The generators $\{s_i\}_{i=0}^k$ satisfy the relations:
  \begin{align*}
		s_i^2 &= 1 \textrm{ for } 1\leq i \leq k\\
		s_is_j &= s_js_i   \textrm{ if } i-j \neq \pm 1\\
		s_is_{i+1}s_i &= s_{i+1}s_is_{i+1}  \textrm{ for } 1\leq i \leq k
			\end{align*}
 where $i-j$ and $i+1$ are understood to be taken modulo $k+1$.

For the purposes of uniformity, if $\eta = \sum_{i=1}^k \eta_i \Lambda_i$, then let
	$\Lambda_0 = \Lambda_{k+1} = {\bf 0}$ be a virtual vector with coefficient $(1 - \sum_{i=1}^k \eta_i)$.
	We can describe the action of $s_i$ on any weight $\eta = (1 - \sum_{i=1}^k \eta_i)\Lambda_0+\sum_{i=1}^k \eta_i \Lambda_i$ using the action of $s_i$ on a dominant weight $\Lambda_j$: %(for $i \neq 0$)

	\begin{equation} \label{eq:siactionlambdaj}
	 s_i( \Lambda_j ) = \begin{cases}
	\Lambda_j&\hbox{ if }j\neq i\\
	\Lambda_{i+1} - \Lambda_i + \Lambda_{i-1}&\hbox{ if }i=j.%\hbox{ and }1\leq i\leq k
	\end{cases}
	\end{equation}

A \deft{reduced word} for an element $w \in \widetilde{\A}_k$ is a minimal word for $w$ in the simple 
reflections; the \deft{length} $\ell(w)$ is the length of any such minimal word.  
The group $\widetilde{\A}_k$ acts on the vector space $V$ and acts faithfully on the fundamental alcove $\Anull$.  
%The action of $\widetilde{A}_k$ on $V$ is faithful on $\Anull$, that is, 
%$w \Anull  \neq v  \Anull$ for $w \neq v$.

For $w \in \widetilde{\A}_k$, we define $A_w := w \Anull$ (so that $A_{uv}=u A_v= uv\Anull$). 
The $A_w$ are called \deft{alcoves}; 
the union of all alcoves is $V$.  The \deft{inversion set} $\inv(w)$ of $w$ is the set of hyperplanes 
that lie between $A_w$ and $\Anull$ (see Corollaries 1.4.4 and 1.4.5 of~\cite{BB}); then $\ell(w)=|\inv(w)|$.  
The \deft{weak order} is a partial order on $\widetilde{\A}_k$ defined by 
$w \leq u$ iff $\inv(w) \subseteq \inv(u)$ (Proposition 3.1.3 of~\cite{BB}).

Just as with finite permutations, it will be useful for computations to faithfully model the action of the
simple reflections on the affine symmetric group.  
In this way, the set of \deft{affine permutations} is the set of bijections 
\[w:\mathbb{Z}\to \mathbb{Z} \text{ such that } \sum_{i=1}^{k+1} w(i)=\binom{k+2}{2} 
\text{ and } w(i+(k+1))=w(i)+(k+1).\]  Such permutations are specified by their values on the ``window'' $1,2,\ldots,k+1.$
We choose our convention for the action on integers by assuming that if
$w(a) = i$ and $w(b) = i+1$, then 
$(s_i w)(a) = i+1$ and $(s_i w)(b) = i$.
We may therefore refer to an affine permutation $w$ using the \deft{one-line notation} $[w(1),w(2),\cdots,w(k+1)]$
because these $k+1$ values determine the action of $w$ on all integers.  We will also use the notation 
$$((i,j)) = s_i s_{i+1} \cdots s_{j-2} s_{j-1} s_{j-2} \cdots s_{i+1} s_i$$
to represent the affine symmetric group element which interchanges $i$ and $j$ in the affine permutation $w$ 
(where the indices of the reflections are taken mod $(k+1)$).
Note that $1\leq i \leq k-b+1$, the transposition $((i,i+b+a(k+1)))$ is the reflection
across the hyperplane $H_{\alpha_i + \alpha_{i+1} + \cdots + \alpha_{i+b-1}, -a }$.

Let $\phi(s_i) := s_{i+1},$ where the index is taken mod $k+1$, and extend $\phi$ to elements of $\A_k$ by acting on a reduced word $\mathbf{w} = s_{i_1} s_{i_2} \cdots s_{i_r}$ by 
$\phi(\mathbf{w}) := \phi(s_{i_1}) \phi(s_{i_2}) \cdots \phi(s_{i_r})$.  

\begin{prop}
	If $\mathbf{w}$ and $\mathbf{w'}$ are equivalent as reduced words in $\widetilde{\A}_k$, 
	then $\phi{\mathbf{w}}$ and $\phi{\mathbf{w'}}$ are also equivalent as reduced words in $\widetilde{\A}_k$.
\end{prop}

\begin{proof}
This follows from Tit's lemma that any two reduced words are connected by braid 
moves---if a sequence of braid moves connects the words $\mathbf{w}$ and $\mathbf{w'}$, 
then the image of these braid moves connects $\phi{\mathbf{w}}$ and $\phi{\mathbf{w'}}$.
\end{proof}

We can therefore refer to $\phi(w)$ as the image of $w$ under this reindexing.  We also allow $\phi$ to act on weights by $\phi(\Lambda_j):=\Lambda_{j+1}$.  For any weight $\eta$, $\phi(s_i)\phi(\eta)=\phi(s_i \eta),$ where we incorporate $\Lambda_0$ as above.

An alcove $A_w$ lies in the dominant chamber $C$ if and only if $w$ is a minimal length coset representative of $\widetilde{\A}_k/\A_k$~\cite{H}.  The set of minimal length (right) coset 
representatives of $\widetilde{\A}_k/\A_k$ is denoted $\Acos$.

\begin{defn}
\label{def:affinegrassmannian}
A permutation $w \in \Acos$ is an \deft{affine Grassmannian permutation}.  
\end{defn}

In one-line notation, the affine Grassmannian permutations are the increasing permutations---those $w$ such that $w(i)<w(i+1)$ for $1 \leq i \leq k+1$.

An alcove may be specified by its $k+1$ vertices, which---as the intersections of any $k$ of its 
facets---are weights.  We can describe the action of a generator of $\widetilde{\A}_k$ on the 
vertices of an alcove as follows.  To each weight $\eta = \sum_{i=1}^{k} \eta_i \Lambda_i$, 
we associate the label $\lbl(\eta) = (\sum_{i=1}^{k} i \eta_i)\mod{(k+1)}$. 
Then every alcove contains exactly one weight of every label in $\{0,1, \ldots, k\}$.

\begin{prop}[Lemma 6.1~\cite{S}]
\label{prop:oneone}
Suppose $A_w$ has vertices 
$v_0, v_1, \dots, v_k$ with $\lbl(v_j) = j$.  Then $A_{w s_i}$ is the alcove which has vertices 
$\{ v_j : j\neq i\}$ and the vertex obtained by reflecting $v_i$ across the affine hyperplane 
spanned by $\{v_j:j\ne i \}$.
\end{prop}

%on avector $\eta = \sum_{i=1}^k \eta_i \Lambda_i$ can be determined by the action of thesimple reflections on

%where we let $\Lambda_0$ be a virtual vector with coefficient 
%where we take $\Lambda_0 = \Lambda_{k+1} = 0$.  The action of $s_0$ is consistent with these formulas---if we let $\Lambda_0$ be a virtual
%vector with coefficient $(1 - \sum_{i=1}^k \eta_i)$, we define
%$$s_0( \eta ) = \eta + \left(1 - \sum_{i=1}^k \eta_i\right) (\Lambda_1 +  \Lambda_k)~.$$

%so it gives a precise formula about how a set of $k+1$ vertices bounding an alcove `walk' through the lattice under the action of
%a word in the generators.%given a reduced word $w = s_{i_1} s_{i_2} \cdots s_{i_r}$

Define a (non-stuttering) \deft{gallery} of alcoves to be a sequence $(A_{w_1},A_{w_2},\ldots,A_{w_n})$ 
such that $A_{w_i}$ and $A_{w_{i+1}}$ are distinct and share a facet (see e.g. \cite{Garrett}).
Proposition~\ref{prop:oneone} 
gives a bijection between galleries of alcoves and words in the simple generators $\{s_i\}_{i=0}^k$ 
which are equivalent to $w_n$.  Given a reduced word for $w \in \widetilde{\A}_k$, one can therefore 
efficiently determine the location of $A_w$, as illustrated in Figure~\ref{eg:walk}.
% illustrates how alcove walks can be used to compute the bijection between elements of $A_k$ and alcoves in $V$.  %This bijection  elements of $W^0$ and alcoves in $C$).

%\begin{eg}

  %The vertices of the alcoves
%are weights $\eta$ and their labels $\lbl(\eta)$ are indicated in the diagram.  
\begin{figure}[htbp]
\begin{center}
\includegraphics[width=3in]{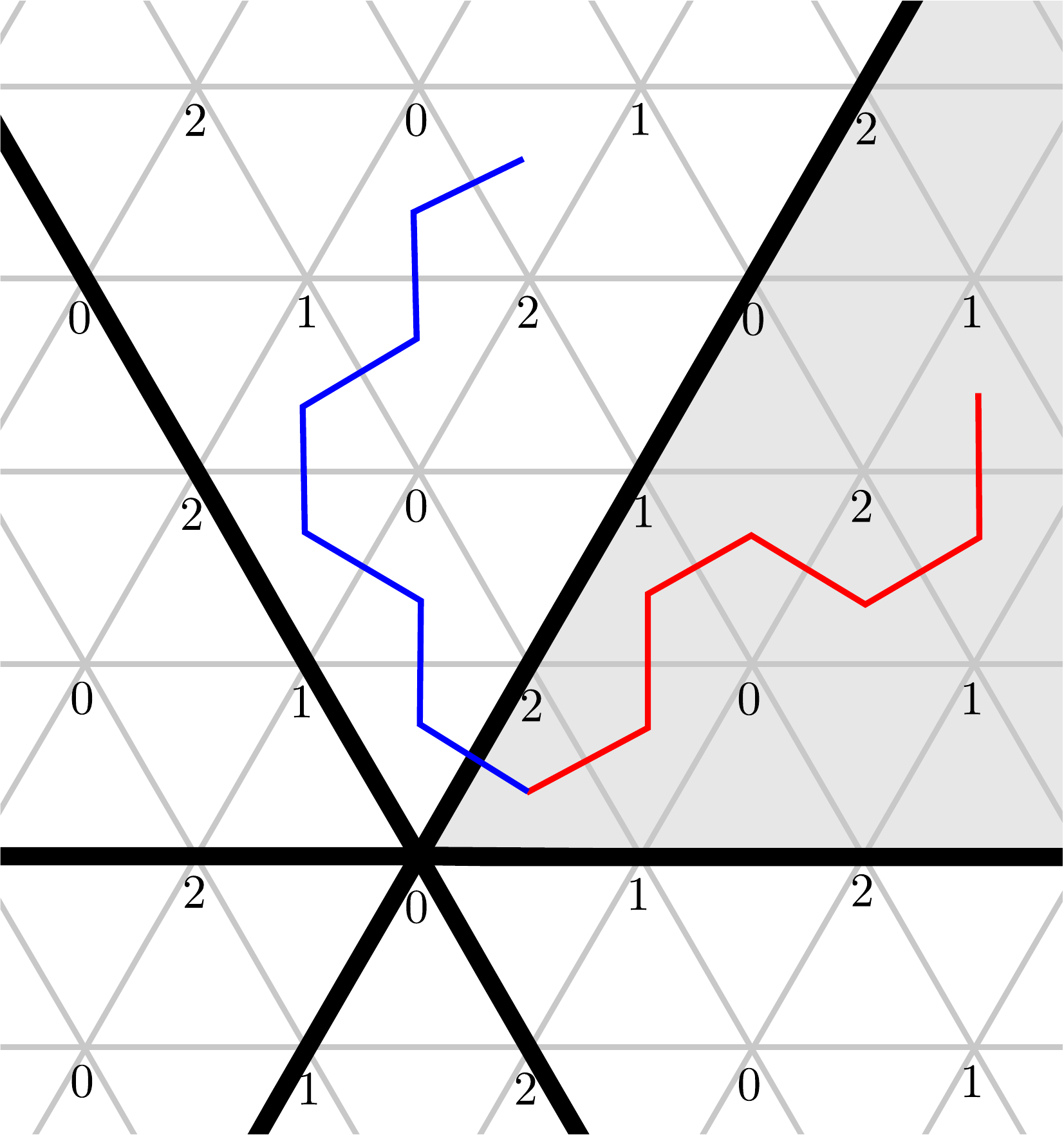}\hskip .2in\includegraphics[width=.5in]{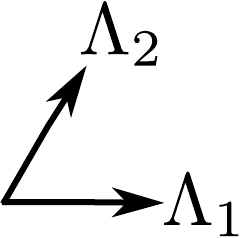}
\end{center}
\caption{The affine arrangement for $\widetilde{\A}_k$ with two example galleries each of which start at $\Anull$.  
The thick black lines are the hyperplanes $\{H_{\alpha} : \alpha \in \Phi\}$, the region shaded in gray is 
the  dominant  chamber,  and  each  weight  $\eta$  is  marked  with  its  label  $\lbl(\eta)$.    The  red  gallery 
terminates  in  the  dominant  alcove  $A_{s_0  s_1  s_2  s_1  s_0  s_1}$  and  remains  in  the  dominant  chamber, 
while the blue gallery ends at $A_{s_1 s_0 s_2 s_1 s_2 s_0 s_1}$ and leaves the dominant chamber in the 
second facet of the gallery.}%   s_0 s_1 s_2 s_1 s_0 s_1     | s_1 s_0 s_2 s_1 s_2 s_0 s_1
\label{eg:walk}
\end{figure}

%The hyperplanes $H_{\alpha_i, n}$ will intersect with $A_w$ 
%either as the empty set,
%at a single weight,
%or in a facet of the alcove (the convex hull of $k$ of the vertices of $A_w$).
%%%  MZ: I was looking at the 3d picture of this and I am pretty sure that this is false
%%%  If you look, in 3d it is possible to intersect in a facet, a pair of vertices, a single vertex or not at all.
%%%  I imagine that in general it is possible to intersect in any $r$ vertices where 1<=r<=k.
%using a gallery of alcoves and the fact that the alcoves $A_{s_i w}$ and $A_w$ share a facet.

%%%%%%%%%%%%%%%%%%%%%%%%%%%%%%%%%%%%%%%%%%%%%%%%%%%%%%%%%%%%%%%%%%%%%%%%%%%%%%%%%%%%%%%%%%
%\subsection{$(k+1)$-cores and $k$-bounded partitions}
%\label{sec:corespartitions}
%%%%%%%%%%%%%%%%%%%%%%%%%%%%%%%%%%%%%%%%%%%%%%%%%%%%%%%%%%%%%%%%%%%%%%%%%%%%%%%%%%%%%%%%%%

%%%%%%%%%%%%%%%%%%%%%%%%%%%%%%%%%%%%%%%%%%%%%%%%%%%%%%%%%%%%%%%%%%%%%%%%%%%%%%%%%%%%%%%%%%
\subsection{$(k+1)$-cores}
%%%%%%%%%%%%%%%%%%%%%%%%%%%%%%%%%%%%%%%%%%%%%%%%%%%%%%%%%%%%%%%%%%%%%%%%%%%%%%%%%%%%%%%%%%
A \deft{partition} $\lambda$ is a finite non-increasing sequence of positive integers 
$\lambda=(\lambda_1 \geq \lambda_2 \geq \cdots \geq \lambda_n)$.  A partition $\lambda$ 
has an associated \deft{Young diagram} of cells 
\[\{(i,j) : 1 \leq j \leq \lambda_i, 1 \leq i \leq n\} \subset \mathbb{N}\times \mathbb{N},\] 
which we will freely associate with $\lambda$ itself (the cell $(1,1)$ is drawn at the bottom left of the diagram).
\deft{Young's lattice} is the set of all partitions ordered by inclusion of Young diagram.  
The \deft{transpose} of $\lambda$ is the partition 
$\lambda':=\left(|\{ \lambda_i \geq j : 1\leq i \leq n\}| \right)_{j=1}^{\lambda_1}$.  
The \deft{hook} of the cell $(i,j)$ is $\hook_\lambda(i,j) := \lambda_i + \lambda_j' - i -j +1$.

\begin{defn}
	\label{def:ckp}
A partition $\core$ is called a \deft{$(k+1)$-core} if $\core$ has no hook of size a multiple of $(k+1)$.  
Let $\Ckp$ denote the set of all $(k+1)$-cores.
\end{defn}

The affine symmetric group $\widetilde{\A}_k$ has an action on $\Ckp$, defined as follows.  Let the 
\deft{content} of a cell $(i,j)$ be the integer $(j-i) \mod k+1$. We call a cell \deft{$i$-addable} 
(resp. \deft{$i$-removable}) if $\core$ with (resp. without) the cell is a partition.  Let 
$\cella_i(\core)$ be the set of all addable cells for $\core$, and let $\cellr_i(\core)$ be the set 
of its removable cells.  It is not hard to see (using, for example, the abacus model) that one cannot have 
$\cella_i(\core) \neq \emptyset$ and $\cellr_i(\core) \neq \emptyset$ simultaneously.

If $\core \in \Ckp$, then we let the generators $\{s_i\}_{i=0}^k$ of $\widetilde{\A}_k$ act by

\begin{equation*}
\core \cdot s_i   = \begin{cases}
\core \cup \cella_i(\core) & \hbox{ if } \cella_i(\core) \neq \emptyset\\
\core \setminus \cellr_i(\core) & \hbox { if } \cellr_i(\core) \neq \emptyset \\
\core & \hbox { otherwise.}
\end{cases}
\end{equation*}

\begin{remark}  We make the action on $(k+1)$-cores and (and similarly for $k$ bounded partitions in the next section)
a right action with the notation $\core \cdot w$.  This is different than notation in other references 
(e.g. \cite{LM2, LLMSSZ, Lam2}),
however this is because we are considering both left and right actions on many different objects
and we want to ensure that they all agree.
By assuming that the left action of $\widetilde{\A}_k$ is reflection across hyperplanes and then the right action 
on alcoves is given in Proposition \ref{prop:oneone} and should follow notation for the action on cores.
\end{remark}

\begin{thm}[Proposition 40~\cite{LM2}, see also \cite{Lascoux}]
\label{thm:core_to_grassmannian}
If $\core \in \Ckp$, then $\core \cdot s_i \in \Ckp$.  The action of $\widetilde{\A}_k$ induces an order-preserving bijection between $\Ckp$ and the affine Grassmannian permutations $\Acos$.
\end{thm}

Call these two bijections $\mathfrak{r}: \Ckp \to \Acos$ and $\mathfrak{c}=\mathfrak{r}^{-1}.$

\begin{eg}
	
If $k=4$, then we compute the $5$-core corresponding to 
$w=  s_0 s_1 s_2 s_3 s_4 s_0 s_3 s_2  \in  \widetilde{\A}_4^{0}$  by %s_0 s_1 s_2 s_3 s_4 s_0 s_3 s_2
%is a minimal length coset representative of $W/W_0$.
%$ \emptyset \cdot s_0 = (1), (1) \cdot s_1 = (2), (2) \cdot s_2 = (3),  (3) \cdot s_3 = (4),  (4) \cdot s_4 = (5,1),
% (5,1) \cdot s_0 = (6,2), (6,2) \cdot s_3 = (6,2,1),  (6,2,1) \cdot s_2 = (6,2,1,1)$.  Visually we have

\squaresize=10pt
\[
\emptyset \rightarrow \young{0\cr} \rightarrow \young{&1\cr} \rightarrow \young{&&2\cr}
\rightarrow \young{&&&3\cr}\rightarrow \young{4\cr&&&&4\cr}
\]
\[
\rightarrow \young{&0\cr&&&&&0\cr}\rightarrow \young{3\cr&\cr&&&&&\cr}\rightarrow \young{2\cr\cr&\cr&&&&&\cr}
\]
\end{eg}

%%%%%%%%%%%%%%%%%%%%%%%%%%%%%%%%%%%%%%%%%%%%%%%%%%%%%%%%%%%%%%%%%%%%%%%%%%%%%%%%%%%%%%%%%%
\subsection{$k$-bounded partitions}
\label{sec:kbounded}
%%%%%%%%%%%%%%%%%%%%%%%%%%%%%%%%%%%%%%%%%%%%%%%%%%%%%%%%%%%%%%%%%%%%%%%%%%%%%%%%%%%%%%%%%%

\begin{defn}
A partition $\bndd$ is called \deft{$k$-bounded} if $\bndd_i \leq k$ for all $1 \leq i \leq \ell(\bndd)$.  
Let $\Pk$ be the set of all $k$-bounded partitions.
\end{defn}

There is a simple bijection from $\Ckp$ to $\Pk$---the $(k+1)$-core $\core \in \Ckp$ is sent to the $k$-bounded partition
$\mathfrak{p}(\core) := (\mathfrak{p}(\core)_i)_{i=1}^n,$ where 
$\mathfrak{p}(\core)_i := |\{ (i,j) \in \core : \hook_\core(i,j)\leq k\}|$.  
The reverse bijection takes the $k$-bounded partition $\bndd \in \Pk$ and produces the 
$(k+1)$-core $\mathfrak{c}(\bndd)$ by pushing the rows of $\bndd \in \Pk$ to the right 
until the only cells with hook-lengths less than or equal to $k$ are those that were originally in $\bndd$
(a more detailed description may be found in \cite{LM2} or \cite{LLMSSZ}).

In~\cite{LM2}, L.\ Lapointe and J.\ Morse studied a poset on $k$-bounded partitions arising 
from their $k$-Pieri rule with A.\ Lascoux~\cite{LLM}.  They introduced the order as a $k$-analogue 
of Young's lattice---it is the poset on $\Pk$ generated by the covering relations 
$\bndd_1 \lessdot \bndd_2$ if $\bndd_1 \subseteq \bndd_2,$ $\bndd_1^{\omega_k} \subseteq \bndd_2^{\omega_k}$, 
and $|\bndd_2| - |\bndd_1|=1,$ where $\bndd^{\omega_k}:=\mathfrak{p}(\mathfrak{c}(\bndd)')$ is the \deft{$k$-conjugate}.
%For a $(k+1)$-core $\core$, we let $\mathfrak{p}(\core)$ denote the corresponding 
%$k$-bounded partition, and we will let $\mathfrak{c}$ denote the inverse map.

\begin{thm}[Theorem 7, Corollary 25~\cite{LM2}]
\label{thm:core_to_bounded}
The maps $\mathfrak{c}$ and $\mathfrak{p}$ are order-preserving bijections between the lattice on $\Ckp$ and $\Pk$.
\end{thm}

Note that the $k$-bounded partitions therefore inherit an action of $\widetilde{\A}_k$ from the $(k+1)$-cores.

%\begin{eg}
%$\core = (6,2,1,1)$ is a $5$-core and leftmost two cells in the longest row of the diagram for $\lambda$
%$$\young{\cr\cr&\cr\gsqr&\gsqr&&&&\cr}$$
%have a hook greater than or equal to $5$, hence $\mathfrak{p}(6,2,1,1) = (4,2,1,1)$ and $\mathfrak{c}(4,2,1,1) = (6,2,1,1)$.
%\end{eg}

\begin{eg}\label{eg:pushout}
For an example of the maps $\mathfrak{c} : \Pk \rightarrow \Ckp$ and
$\mathfrak{p} : \Ckp \rightarrow \Pk$ which concerns us in the next sections,
consider $k+1 = 5$ and $\bndd = (4^2,3^2,1^4)$ which is a concatenation of 
a sequence of rectangles each with a hook length of $4$.  The algorithm for computing $\mathfrak{c}(\mu)$
says that the first $4$ rows are pushed $1$ cell left, then 
\begin{itemize} 
\item first $2$ rows pushed $3$ cells left
\item first row is pushed $4$ cells left
\end{itemize}
The final core diagram has the following Ferrer's diagram (with the cells with hook-length greater than $4$ shaded).
\[
\young{
\cr
\cr
\cr
\cr
\gsqr&&&\cr
\gsqr&&&\cr
\gsqr&\gsqr&\gsqr&\gsqr&&&&\cr
\gsqr&\gsqr&\gsqr&\gsqr&\gsqr&\gsqr&\gsqr&\gsqr&&&&\cr}
\]
Therefore $\mathfrak{c}(4^2,3^2,1^4) = (12,8,4,4,1,1,1,1) = \core$.  The unshaded cells of the
core are those
with hook-length less than $5$.  They form a skew shape consisting of disjoint rectangles, and the $4$ bounded partition $\bndd$ can be  computed from the $5$-core
$\core$ by deleting the shaded cells and left justifying them.
\end{eg}

We can pass directly from $k$-bounded partitions to
affine Grassmannian permutations as follows.  Given a $k$-bounded partition $\bndd$, fill the cells $(i,j)$ with
the simple generator $s_{j-i \mod (k+1)}$.  The bijection is given by forming the element 
$\mathfrak{r}(\bndd)$ with reduced word obtained from reading the rows of the diagram from
left to right and bottom to top (noting that our partition diagrams use the French convention with the
largest row on the bottom).  Since the word is reduced, if we let the \deft{length} of 
$\bndd$ be $\ell(\bndd):=|\{ (i,j) \in \bndd\}|$, then $\ell(\bndd)=\ell(\mathfrak{r}(\bndd))$.  
We abbreviate the alcove $A_{\mathfrak{r}(\bndd)}$ as $A_{\bndd}$.%=\ell(\core)$.

\begin{thm}[Corollary 48~\cite{LM2}]
\label{thm:bounded_to_grassmannian}
There is an order-preserving bijection between the lattice on $\Pk$ and the affine Grassmannian 
permutations $\Acos$.
\end{thm}

Call the two bijections $\mathfrak{r}: \Pk \to \Acos$ 
and $\mathfrak{p}=\mathfrak{r}^{-1}.$

\begin{eg}
Let $k=4$, then the $4$-bounded partition $(4,2,1,1)$ has a diagram which we fill the $(i,j)$ cell
with the element $s_{j-i \mod 5}$ as
$$\young{s_2\cr s_3\cr s_4&s_0\cr s_0&s_1&s_2&s_3\cr}~.$$
Then by reading the word of the entries of this tableau, the
partition is identified with $\mathfrak{r}(4,2,1,1) = s_0 s_1 s_2 s_3 s_4 s_0 s_3 s_2$.
\end{eg}

%%%%%%%%%%%%%%%%%%%%%%%%%%%%%%%%%%%%%%%%%%%%%%%%%%%%%%%%%%%%%%%%%%%%%%%%%%%%%%%%%%%%%%%%%%
\section{Suter symmetry}
\label{sec:suter_symmetry}
%%%%%%%%%%%%%%%%%%%%%%%%%%%%%%%%%%%%%%%%%%%%%%%%%%%%%%%%%%%%%%%%%%%%%%%%%%%%%%%%%%%%%%%%%%

For each $k \in \mathbb{N}$, R.\ Suter described a set $Y^k$ of $2^k$ partitions with unexpected 
dihedral symmetries~\cite{su}.  The subset $Y^k$ is an order ideal in Young's lattice, and so to 
specify $Y^k$ it is enough to specify that its maximal elements are exactly those rectangles with 
hook-lengths at most $k$.

\begin{defn}[\cite{su}]
\label{defn:yk}
Let $R_i := (i^{k+1-i})$ for $1 \leq i \leq k$.  Then we define 
\[Y^k := \{ \core : \core \subseteq R_i \text{ for some } 1 \leq i \leq k\}.\]
\end{defn}

R.\ Suter proved that $Y^k$ had an action of the dihedral group of order $2(k+1)$, coming from 
the usual symmetry of partition transposition along with the $(k+1)$-fold cyclic action~\cite{su}:

\[(\core_1,\core_2,\ldots,\core_n) \mapsto (k-\core_1,\core'_1-1,\core'_2-1,\ldots,\core'_{n'}-1)'.\]

This symmetry is illustrated in Figure~\ref{fig:suter3} for $k = 2,3,4$.

\begin{figure}[htbp]
\begin{center}
\includegraphics[width=1.2in]{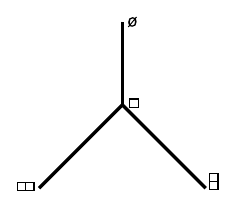} \hskip .3in
\includegraphics[width=2in]{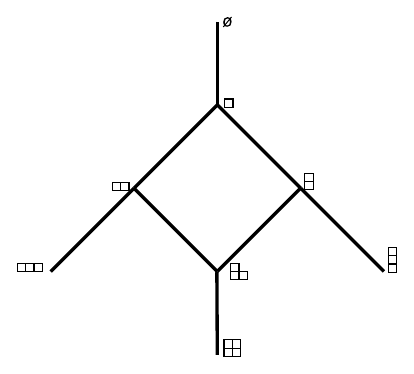}
\end{center}
\begin{center}
\includegraphics[width=3in]{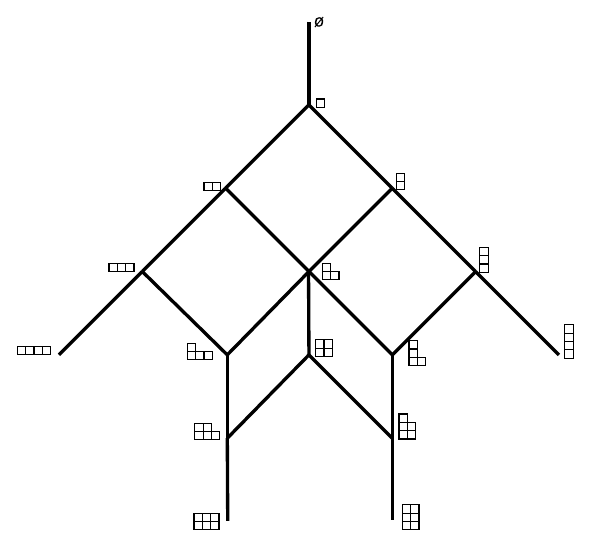}
\end{center}
\caption{Three examples of the $k+1$ dihedral symmetry of $Y^k$ for $k = 2,3,4$.}\label{fig:suter3}
\end{figure}

In~\cite{su2}, R.\ Suter interpreted $Y^k$ as the set of abelian ideals in a Borel subalgebra of 
the Lie algebra of type $\A_k$.  He connected this interpretation with a remarkable result of D.\ Peterson, establishing 
a bijection between $Y^k$ and alcoves in $2\Anull$.  Briefly, the partitions in $Y^k$ correspond to the 
inversion set of the affine Grassmannian permutation.  As the fundamental alcove has a $(k+1)$-fold 
cyclic symmetry, so does $2\Anull$---and so does $Y^k$.  Thus, the natural geometric symmetry on 
$2\Anull$ explains the unexpected symmetry of $Y^k$.  Figure~\ref{notfigure1} illustrates the 
correspondence between $Y^2$ and $2\Anull$ in type $\widetilde{\A}_2$.  

\begin{figure}[htbp]
\label{notfigure1}
\begin{center}
\includegraphics[width=2in]{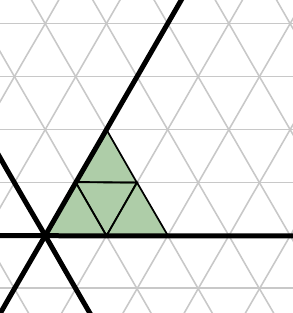}\hskip .2in\includegraphics[width=.5in]{direction2.pdf}
\end{center}
\caption{The two-fold dilation $2\Anull$ of the fundamental alcove of $\widetilde{\A}_2$.  The
alcoves within $2\Anull$ are in bijection with $Y^2=\{ \emptyset, (1), (2), (1,1) \}$.}
\end{figure}

Using the constructions of Section~\ref{sec:comb_affine_A}, the bijection betwen $Y^k$ and $2\Anull$ 
proceeds in the following way.  Since any $\core \in Y^k$ has maximum hook-length at most $k$, $\core$ 
may be viewed as a $(k+1)$-core (or $k$-bounded partition) so that $Y^k$ can be equivalently described 
as an order ideal in the lattice of $(k+1)$-cores.  By Theorem~\ref{thm:core_to_grassmannian}, we can 
therefore associate to each partition $\core \in Y^k$ a dominant alcove $A_\core$.

This sends the maximal elements $R_i \in Y^k$ to the alcoves $A_{R_i} \in 2\Anull$ with a facet on 
the hyperplane $H_{\alpha_0,2}:=\{x:\langle \alpha_0,x \rangle =2\}$.  Such an alcove is characterized 
by the unique fundamental weight $\wt_i$ that is a vertex of $A_{R_i}$ not on the hyperplane $H_{\alpha_0,2}$.  
We emphasize this with correspondence with a proposition.

\begin{prop}
\label{prop:rectangles_and_fund_weights}
There is a bijection between the rectangles $R_i$ and the fundamental weights $\wt_{i}$ for $1 \leq i \leq k$.
\end{prop}

This is trivially true because the two sets each have $k$ elements, but we can refine the proposition as follows.

\begin{lemma}
For $1 \leq i \leq k$, $\mathfrak{r}(R_{i}) \Anull = A_{\mathfrak{r}(R_{i})} = \Anull +\wt_{i}.$
\label{lem:pseudo_1}
\end{lemma}
\begin{proof} 

Since $w:=\mathfrak{r}(R_{i})$ contains no simple reflection $s_i$, consider $\phi^{-i}(w)$.  
This element $\phi^{-i}(w)$ contains no simple reflection $s_0$ and is therefore an element of $\A_k$.
We determine $\phi^{-i}(w)$ by reading the rows of the rectangle using the map $\mathfrak{r}$.  The corresponding reduced word has the form
\[
\phi^{-i}(w) =  
(s_{k+1-i} s_{k+2-i} \cdots s_{k})
(s_{k-i} s_{k+1-i} \cdots s_{k-1})
\cdots(s_{2} s_{3} \cdots s_{i+1}) (s_{1} s_{2} \cdots s_{i})   .
\]
%104210
%321 432
% here is an example to make convince me that this is right
% k=4 (k+1=5) and look at the R_3 = (3,3)
% R_3 =
% 4 0 1
% 0 1 2
% r(R_3) = (s_0 s_1 s_2) (s_4 s_0 s_1)
% \phi^{-3}( R_3 ) =
% 1 2 3
% 2 3 4
% phi^{-3}( r( R_3 ) ) = (s_2 s_3 s_4) (s_1 s_2 s_3)
% IF when we act on one line notation the action of s_i is to exchange w_i and w_{i+1} (and not
% to exchange i and i+1 in w
%By acting this reduced word on the window $[1,2,3, \ldots, k+1]$, it will have the representation 
The corresponding permutation will have one-line notation
$[k+2-i,k+3-i,\ldots,k,k+1,1,2,\ldots,k+1-i]$, which has a single (finite) descent in position $i$.
%\todo{ make sure the one line notation is correct.  I suspect it will be the inverse with the current convention. }

The bounding $k$ \emph{finite} hyperplanes of $A_{\phi^{-i}(w)}$ may now be read off from the 
one-line notation from the values of $\phi^{-i}(w)(j)$ and $\phi^{-i}(w)(j+1)$---since
$\phi^{-i}(w)(j+1) = \phi^{-i}(w)(j)+1$ for every $j \neq i,k+1$, 
then the hyperplanes $H_{\alpha_j,0}$ bound the alcove for $j\neq i$. For $j=i$, 
we get the hyperplane $H_{\alpha_0,0}$.  The single \emph{affine} hyperplane corresponds to 
the values of $\phi^{-i}(w)(0)$ and $\phi^{-i}(w)(1)$, which are $-i$ and $k+2-i$ respectively.
Therefore the bounding hyperplanes are
\[\{ H_{\alpha,0} : \alpha \in \Delta, \alpha \neq \alpha_{k+1-i}\} \cup \{ H_{\alpha_0,0}, H_{\alpha_{k+1-i},-1}\}.\]
We must now compute the intersection of any $k$ of these to determine the corresponding weight.  
If we drop any particular $H_{\alpha_j,0}$ from the first part of the union, we specify that the 
intersection is the point $x$ such that $\langle \alpha,x \rangle=0$ for $\alpha \neq \alpha_{k+1-i},\alpha_{j}$, 
$\langle \alpha_0,x \rangle=0$, and $\langle \alpha_{k+1-i},x \rangle=-1.$  The first condition specifies that $x$ 
is a linear combination of $\Lambda_j$ and $\Lambda_{i}$, from which the second and third conditions tell us that 
$x = \Lambda_0-\Lambda_{k+1-i}+\Lambda_{j}$.  Similarly, if we drop $H_{\alpha_0,0},$ then we obtain the point
 $x=2\Lambda_0-\Lambda_{k+1-i}$, and if we drop $H_{\alpha_{k+1-i},-1}$, then we obtain the origin 
$x=\Lambda_0.$  Thus, the vertices of $A_{\phi^{-i}(w)}$ are the weights 
\[\{\Lambda_0-\Lambda_{k+1-i}+\Lambda_{j} : j \neq k+1-i, j\neq 0\} \cup \{2\Lambda_0-\Lambda_{k+1-i}, \Lambda_0\}.\]
The vertices of $A_w$ now follow from applying $\phi^{i}$ to the vertices of $A_{\phi^{-i}(w)}$:
\[\{-\Lambda_{0}+\Lambda_i+\Lambda_{j} : j \neq 0, j\neq i\} \cup \{-\Lambda_{0}+2\Lambda_i, \Lambda_i\}.\]
In particular, since the vertices of $\Anull$ are $\{\Lambda_j : 0 \leq j \leq k\}$, we conclude that 
$\mathfrak{r}(R_{i}) \Anull = A_{\mathfrak{r}(R_{i})} =\Anull+\wt_i.$

%In fact, we show that $\phi^{-i}(\mathfrak{r}(R_{i})) \Lambda_{j-i} = -\Lambda_{-i}+\Lambda_{j-2i}+\Lambda_{0}$, from which the result follows by applying $\phi^i$ to $-\Lambda_{-i}+\Lambda_{j-2i}+\Lambda_{0}$. % But $\phi^{-i}(\mathfrak{r}(R_{i}))$ contains no $s_0$, and so is an element of $A_k$.

%Since $\Lambda_{i-1}$ occurs only in the first row of $\mathfrak{r}(R_{i})$, it is easy to check that $\Lambda_{i-1} \to \Lambda_{k}+\Lambda_i$.

%$\Lambda_{i-2}$ occurs in the second row of of $\mathfrak{r}(R_{i})$, and this second row sends it to $\Lambda_{k-1}-\Lambda_{k}+\Lambda_{i-1}$.  By our computation of $\Lambda_{i-1}$, by linearity, and since the first row doesn't contain $\Lambda_{k-1}-\Lambda_{k}$, we conclude that $\Lambda_{i-2} \to \Lambda_{k-1}+\Lambda_i$.
%...
%\Lambda_{i-1}-->\Lambda_{i-2}-\Lambda_{i-1}+\Lambda_i-->\Lambda_{i-3}-\Lambda_{i-2}+\Lambda_i...

%0123...i-1
%k012...i-2
%i+1....2i-k
\end{proof}

\begin{eg}
Fix $i=1$ and consider $R_1$ in type $\widetilde{\A}_2$.  
Then $w=\mathfrak{r}(R_1)=s_0s_2,$ so that $\phi^{-i}(w)=\phi^{-1}(w)=s_2s_1,$ whose alcove $A_{\phi^{-1}(w)}$ has vertices 
$\{\Lambda_0+\Lambda_1-\Lambda_2,2\Lambda_0-\Lambda_2,\Lambda_0\}$.  
Applying $\phi$ to these vertices, we obtain that the vertices of $A_w$ are 
$\{-\Lambda_0+\Lambda_1+\Lambda_2, -\Lambda_0+2\Lambda_1,\Lambda_1\},$ 
from which we conclude that $A_{\mathfrak{r}(R_{1})} = A_{s_0 s_2} = \Anull +\wt_{1}$.
\end{eg}

%Suter defined a cyclic action on $Y^k$ of order $k+1$, described on the Young diagram of a partition. We will not present this here; our generalization comes from a different description of this cyclic action which we now introduce.
%cyclic action : 
%R.\ Suter described a cyclic action on this subset which

%, the two-fold dilation of the fundamental alcove $\Anull$ in type $\widetilde{A}_k$.

%%%%%%%%%%%%%%%%%%%%%%%%%%%%%%%%%%%%%%%%%%%%%%%%%%%%%%%%%%%%%%%%%%%%%%%%%%%%%%%%%%%%%%%%%%
\section{$k$-Schur functions}%The affine Nil-Coxeter algebra and 
\label{sec:kschur}
%%%%%%%%%%%%%%%%%%%%%%%%%%%%%%%%%%%%%%%%%%%%%%%%%%%%%%%%%%%%%%%%%%%%%%%%%%%%%%%%%%%%%%%%%%

We define T.\ Lam's realization of L.\ Lapointe, A.\ Lascoux, and J.\ Morse's $k$-Schur functions 
in the affine Nil-Coxeter algebra, and provide two relevant theorems.

%%%%%%%%%%%%%%%%%%%%%%%%%%%%%%%%%%%%%%%%%%%%%%%%%%%%%%%%%%%%%%%%%%%%%%%%%%%%%%%%%%%%%%%%%%
\subsection{The affine Nil-Coxeter algebra}
%%%%%%%%%%%%%%%%%%%%%%%%%%%%%%%%%%%%%%%%%%%%%%%%%%%%%%%%%%%%%%%%%%%%%%%%%%%%%%%%%%%%%%%%%%

The affine nilCoxeter algebra~\cite{Lam} $\mathbb{A}_k$ is the algebra generated by $\{u_i\}_{i=0}^k$, with relations: 
  \begin{align*}
		u_i^2 & = 0  \textrm{ for } 0\leq i \leq k \\
		u_iu_j & = u_ju_i   \textrm{ if } i-j \neq \pm 1\\
		u_iu_{i+1}u_i & =u_{i+1}u_iu_{i+1}  \textrm{ for } 0\leq i \leq k
			\end{align*}
 where $i-j$ and $i+1$ are understood to be taken modulo $k+1$.  If $s_{i_1} \dots s_{i_m}$ is a reduced word for an element $w \in \widetilde{\A}_k$, we 
define $\uu(w) = u_{i_1} \dots u_{i_m}$. Then $U := \{ \uu(w): w \in \widetilde{\A}_k\}$ is a basis of $\mathbb{A}_k$.

The images of the affine Grassmannian permutations $\{ \uu(w) : w \in \Acos \}$ 
are non-zero elements of $\mathbb{A}_k$, so that the affine nilCoxeter algebra naturally acts on 
$(k+1)$-cores (and therefore on $k$-bounded partitions): for $\core \in \Ckp$, define 
$\core \cdot u_i = \core \cup \cella_i(\core)$ if $\core$ has at least one addable cell of content $i$, 
and $\core \cdot u_i = 0$ otherwise.

%The affine nilCoxeter algebra has an action on the free abelian group with basis the $(k+1)$-cores.  For $\core \in \Ckp$, define $u_i \core = \core \cup \cella_i(\core)$ if $\core$ has at least one addable cell of content $i$, and $u_i \core = 0$ otherwise.
%such addable cell, and $u_i \nu$ is $0$ otherwise. 
%\begin{equation*}
% u_i \core  = \begin{cases}
%\core \cup \cella_i(\core) & \hbox{ if } \cella_i(\core) \neq \emptyset\\
%0 & \hbox { otherwise.}
%\end{cases}
%\end{equation*}

%generated by the first $k$ complete homogenous symmetric functions $h_1, h_2, \ldots, h_k$~\cite{Lam}.

Motivated by Stanley symmetric functions, T.\ Lam defined a set of elements within $\mathbb{A}_k$ 
that generate a subalgebra isomorphic to a natural subring of symmetric functions~\cite{Lam}.  
We require the following definitions, mimicking the construction of the usual Stanley symmetry 
functions.  An element $u = u_{i_1} u_{i_2} \cdots u_{i_m} \in U$ is said to be \deft{cyclically 
increasing} if each of $i_1, i_2, \ldots, i_m$ are distinct, and whenever $j = i_s$ and $j+1=i_t$ 
then $s<t$ ($j+1$ is taken modulo $k+1$). To a strict subset $D \subset \K$, we let $u_D$ denote 
the unique element of $U$ which is cyclically increasing and is a product of the generators $u_m$ for $m\in D$. 

\begin{defn}[\cite{Lam}]
For $1 \leq i \leq k$, define the element $\bfh_i := \sum_{|D| = i} u_D \in \mathbb{A}_k$.
\end{defn}

\begin{thm}[Corollary 14~\cite{Lam}]
The elements $\{\bfh_i\}_{i=1}^k$ generate a subalgebra isomorphic to the ring
generated by the first $k$ complete homogeneous symmetric functions $h_1, h_2, \ldots, h_k$. 
The isomorphism is defined by identifying $\bfh_i$ and $h_i$.
\end{thm}

%%%%%%%%%%%%%%%%%%%%%%%%%%%%%%%%%%%%%%%%%%%%%%%%%%%%%%%%%%%%%%%%%%%%%%%%%%%%%%%%%%%%%%%%%%
\subsection{$k$-Schur functions}
%%%%%%%%%%%%%%%%%%%%%%%%%%%%%%%%%%%%%%%%%%%%%%%%%%%%%%%%%%%%%%%%%%%%%%%%%%%%%%%%%%%%%%%%%%

The $k$-Schur functions were first introduced by Lapointe, Lascoux and Morse \cite{LLM}, 
who were motivated by the study of positivity Macdonald polynomials.  
They have since appeared in other contexts (see, in particular, \cite{Lam1, Lam2, LS, LM3}).  
We use T.\ Lam's definition of the $k$-Schur functions as elements of $\mathbb{A}_k$.

\begin{defn}[Definition 6.5~\cite{Lam2}]
\label{def:kschur}
Let $s_\emptyset^{(k)} = 1$ and let $\bndd$ be a $k$-bounded partition. 
We inductively construct elements $s_\bndd^{(k)}$ of the subring generated by the 
$\bfh_i$ by defining $s_\bndd^{(k)}$ to be the unique element satisfying the following ($k$-Pieri) rule:
\[ \bfh_i s_\bndd^{(k)} = \sum_\tau s_\tau^{(k)}; \hspace{.5in} .\]
where $\tau = \bndd \cdot \uu(y)$ for some cyclically increasing word $y$ of length $i$.
\end{defn}%inductively defined as the unique element  which s

%\subsection{Expression of rectangle $k$-Schur functions as pseudo-translations} \label{sec:pseudotranslation}

There are two problems involving $k$-Schur functions that naturally arise:

\begin{enumerate}
\item Identify the structure coefficients 
$s_{\nu}^{(k)} s_{\mu}^{(k)} = \sum_\tau c_{\nu \mu}^{\tau(k)} s_{\tau}^{(k)}$, and
\item Understand the expansion $s_\nu^{(k)} = \sum_w a_{\nu w} \uu(w)$  in $\mathbb{A}_k$ 
(the $a_{\nu w}$ are called the $k$-Littlewood-Richardson coefficients).  
\end{enumerate}
	
It turns out that these two problems are related by the equality of coefficients~\cite[Proposition 42]{Lam} 
\[ c_{\nu \mu}^{\tau(k)} = a_{\nu \mathfrak{r}(\tau) \mathfrak{r}(\mu)^{-1}}.\]  
For both problems, several interesting special cases have been solved.  
We present two cases here that will play a role in Section~\ref{sec:gen}.

%The rectangles $R_1, \dots, R_k$ described above play an important role in the study of $k$-Schur functions. 
%$k$-Schur functions, first introduced by Lapointe, Lascoux and Morse \cite{LLM}, were motivated in the study 
%of Macdonald polynomials, but have since appeared in other contexts (see, in
%particular, \cite{Lam1, Lam2, LS, LM3}).

%Each $k$-Schur function $s_\lambda^{(k)}$ is indexed by a $k$-bounded 
%partition $\lambda$ (or equivalently a $(k+1)$-core, or an affine Grassmannian permutation).

%An important open problem in the study of $k$-Schur functions is to understand their multiplication rule. One special case of  is explicitly understood, due to 

The following theorem of L.\ Lapointe and J.\ Morse explicitly identifies the multiplication 
rule when one of the functions is indexed by a rectangle $R_i$.  For two partitions 
$\bndd_1, \bndd_2$, we let $\bndd_1 \cup \bndd_2$ denote the partition obtained by 
combining the parts of $\bndd_1$ and $\bndd_2$ and placing them into non-increasing order.  

\begin{thm}[Theorem 40~\cite{LM3}]
\label{thm:LM}
For a rectangle $R_i$ and a $k$-bounded partition $\bndd$, \[s_\bndd^{(k)} s_{R_i}^{(k)} = s_{\bndd \cup R_i} ^{(k)}.\] 
\end{thm}

%In general, expanding $s_\lambda^{(k)} = \sum_w c_w \uu(w)$ is an open problem.  It has been shown to be equivalent to understanding the structure coefficients of $k$-Schur functions (called the $k$-Littlewood-Richardson coefficients).

In \cite{BBTZ}, the authors explicitly described the expansion of the rectangular 
$k$-Schur functions $s_{R_i}$ as an element in $\mathbb{A}$.  For $\wt$ a weight, 
we say $v \in \widetilde{\A}_k$ is a \deft{pseudo-translation} of $A_w$ in the 
direction $\wt$ if $A_{wv}  =  A_w + \wt$.  We let $\tr_\wt$ 
denote the pseudo-translation of $\Anull$ in the direction $\wt$, so that $\tr_\wt \Anull = \Anull + \wt$.  Pseudo-translations 
have since been realized by T.\ Lam and M.\ Shimozono as being translations of the extended 
affine Weyl group~\cite{LS2}.

%when it is indexed by one of the maximal rectangles $R_1, \dots, R_k$.
%In \cite{BBTZ}, the authors introduced the notion of a pseudo-translation in order to describe the expansion of 
%$k$-Schur functions as an element in $\mathbb{A}$ when it is indexed by one of the maximal rectangles $R_1, \dots, R_k$. 
%\begin{defn}
%In this case we will use $\tr_\eta$ to denote $y$.
%\end{defn}

\begin{thm}[Theorem 4.12~\cite{BBTZ}]\label{thm:bbtz} Inside $\mathbb{A}_k$,
\[ s_{R_i}^{(k)} = \sum_{\Lambda \in \A_k \wt_i} \uu({\tr_\wt}).\]
where $\A_k \wt_i := \{w \wt_i : w \in \A_k\}$.
\end{thm}

Observe that this theorem is an algebraic analogue of the bijection between $R_i$ and 
$\wt_i$ in Proposition~\ref{prop:rectangles_and_fund_weights}, since it states that the 
$k$-Schur functions indexed by $R_i$ correspond to a sum over $\A_k$-orbits of $\Lambda_i$.

%\begin{remark}
%\end{remark}

%%%%%%%%%%%%%%%%%%%%%%%%%%%%%%%%%%%%%%%%%%%%%%%%%%%%%%%%%%%%%%%%%%%%%%%%%%%%%%%%%%%%%%%%%%
\section{Generalized Suter symmetry}
\label{sec:gen}
%%%%%%%%%%%%%%%%%%%%%%%%%%%%%%%%%%%%%%%%%%%%%%%%%%%%%%%%%%%%%%%%%%%%%%%%%%%%%%%%%%%%%%%%%%

From our discussion in Section~\ref{sec:suter_symmetry}, R.\ Suter's set of partitions 
$Y^k$ is understood geometrically as the 2-fold dilation of the fundamental alcove.  
It makes sense to generalize this by instead considering an $m$-fold dilation of $\Anull$.  
Special cases of a translate of this dilation arise in the context of Catalan combinatorics. 
%, whereby its alcoves index parking functions.
% taking all $(k+1)$-cores $\core$ such that the alcove $A_\core$ is contained in $(m+1)\Anull$. 

\begin{defn}
\label{def:ymk}
Let $Y_m^k := \{ \bndd \in \Pk : A_\bndd \in m\Anull\}.$
\end{defn}

By construction, the geometric symmetry on $m\Anull$ induces a symmetry on $Y_m^k$, 
generalizing R.\ Suter's construction when $m=2$.  This geometric symmetry is illustrated in Figure~\ref{fig:3dkeq}.

\begin{figure}[htbp]
\begin{center}
\includegraphics[width=2.5in]{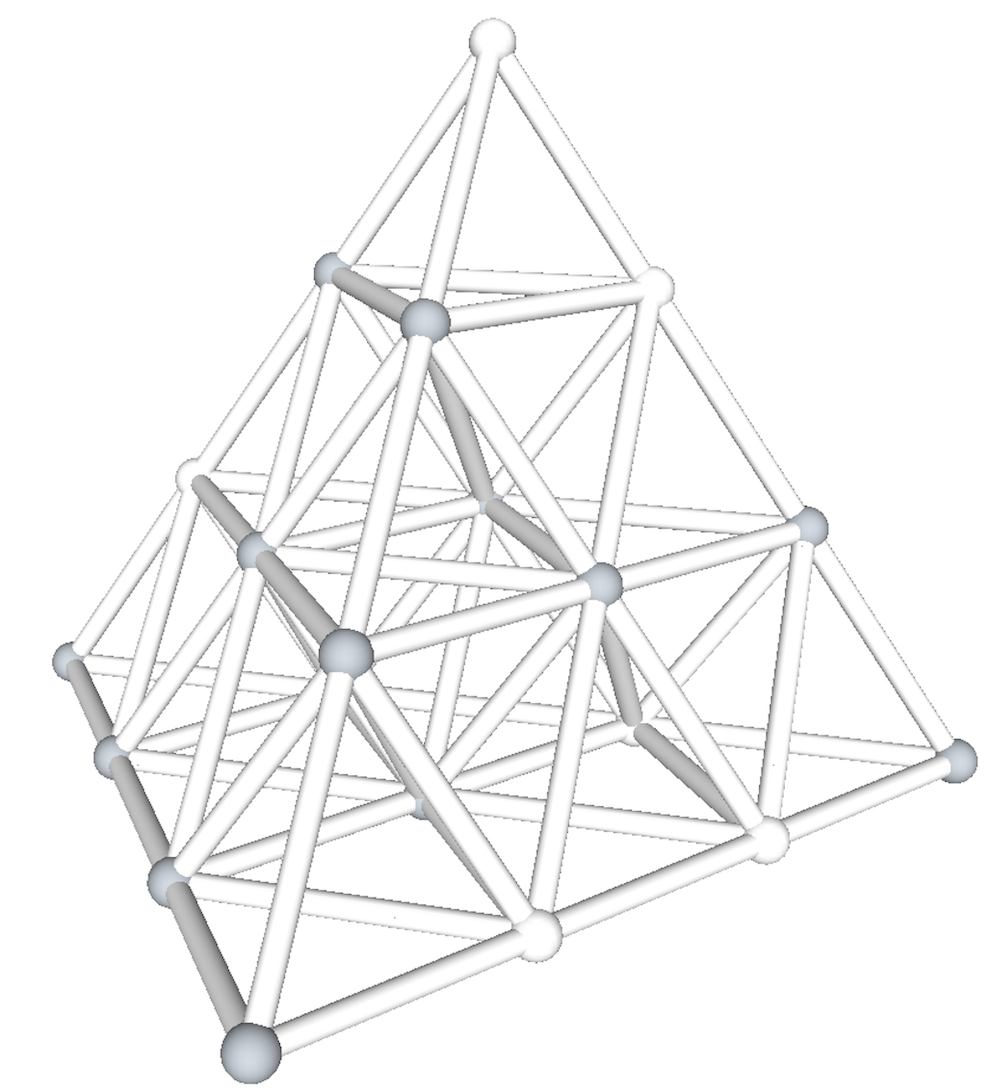}
\end{center}
\caption{The dilation of the fundamental alcove of $\widetilde{\A}_3$ by a factor of $3$
is referred to as $3 \Anull$. Each alcove is in bijection with one of the $4$-cores drawn in Figure~\ref{eg:Y33}.}
\label{fig:3dkeq}
\end{figure}

The problem, then, is to give an intrinsic characterization of the $k$-bounded partitions 
in $Y_m^k$, similar to the description of $Y^k$ in Definition~\ref{defn:yk}.  Exactly as in the case of $Y^k$---since the bijections between $k$-bounded partitions, 
$(k+1)$-cores, and alcoves are order preserving---$Y_m^k$ is an order ideal in the lattice 
of $k$-bounded partitions and it is enough to characterize the maximal elements.  
The maximal alcoves are those with a facet on the hyperplane $H_{\alpha_0,m}$, 
and are again in bijection with the weight that is their unique vertex not on $H_{\alpha_0,m}$.  
These weights are all dominant weights that lie on $H_{\alpha_0,m-1}$ and are simply all 
sums of any $m-1$ fundamental weights $\sum_{j=1}^{m-1} \wt_{i_j}$.

\begin{lemma} \label{lem:action}
For a weight $\eta = \sum_{i=1}^k \eta_i \Lambda_i$ and $|\eta| = \sum_{i=1}^k \eta_i$, 
let $\tr_{\eta}$ be the pseudo-translation $\tr_{\eta}\Anull=\Anull+\eta$.
Then $A_{\phi^p(\tr_{\eta})} = \Anull + c^{p} \eta$
where $c = s_1 s_2 \cdots s_k$.  
In particular, \[\{\phi^{p}(\tr_{\eta}) : 0 \leq p \leq k\} \subseteq \{\tr_{v \eta} : v \in \A_k\}.\]
\label{lem:contains_cycles}
\end{lemma}

% example calculation
% k=3, c= s_1 s_2 s_3
% c(\Lambda_1) = s_1 s_2( \Lambda_1 ) = \Lambda_2 - \Lambda_1 + \Lambda_0
% c(\Lambda_2) = s_1 s_2( \Lambda_2 ) = \Lambda_3 - \Lambda_1 + \Lambda_0
% c(\Lambda_3) = s_1 s_2(\Lambda_0-\Lambda_3+\Lambda_2) = s_1(\Lambda_0-\Lambda_2+\Lambda_1) = 2\Lambda_0-\Lambda_1
% c^2(\Lambda_3) = \Lambda_0-\Lambda_2+\Lambda_1

%k=2 example
%$$
%c \Lambda_1 = (s_1 s_2) \Lambda_1 = s_1 \Lambda_1 = \Lambda_2 - \Lambda_1 + \Lambda_0$$
%$$
%c \Lambda_2 = (s_1 s_2) \Lambda_2 = s_1( \Lambda_0 - \Lambda_2 + \Lambda_1 ) 
%= \Lambda_0 - \Lambda_2 + (\Lambda_0 - \Lambda_1 + \Lambda_2) = \Lambda_0 - \Lambda_1 + \Lambda_0$$

\begin{proof}
Fix the long cycle $c=s_1s_2\cdots s_k.$  We claim that $\tr_{c^p \eta} 
= \phi^{p}(\tr_{\eta})$ for $0\leq p \leq k.$  This is trivially true for $p=0$.
For $p>0$, a computation by induction using
the action of $s_i$ on weights given in equation \eqref{eq:siactionlambdaj} shows that 
$c^p \wt_i = \wt_0-\wt_p+\wt_{i+p}.$  From this we conclude that
\[ c^p \eta = \wt_0-|\eta| \wt_p + \sum_{i=1}^k \eta_i \wt_{i+p}  \]
where the indices of the weights are all taken mod $k+1$.

It remains to show that 
$\phi^{p}(\tr_{\eta})$ is a pseudo-translation of $\Anull$ in the direction $c^p \eta$, 
which we will observe by computing its action on each of the fundamental weights.  
The vertices of $A_{\tr_{\eta}} = \Anull + \eta$ are the 
weights $\{ \wt_j - |\eta| \wt_0 + \eta : 0 \leq j \leq k \}.$  
The vertices of $A_{\phi^{p}(\tr_{\eta})}$ follow from applying $\phi^{p}$ to these:
\[\left\{  \wt_{j+p} - |\eta| \Lambda_p + \sum_{i=1}^k \eta_i \wt_{i+p} : 0 \leq j \leq k \right\}.\]
The bounding vertices of $A_{\phi^{p}(\tr_{\eta})}$ are therefore
$\wt_j - \wt_0 + c^p \eta$ for $0 \leq j \leq k$ and hence we conclude that 
$\phi^p( A_{\tr_{\eta}} ) = 
A_{\phi^p(\tr_{\eta})} = \phi^{p}(\tr_{\eta}) \Anull=\Anull+c^p \eta$, so that $\phi^p(\tr_{\eta}) =\tr_{c^p \eta}$.
\end{proof}

%One conclusion to draw from this last lemma is the action of $\phi$ on the alcoves indexed by the
%elements $\mathfrak{r}(R_i)$.
We visualize the action of $\phi$ on the alcoves of $2\Anull$ for $\widetilde{A}_2$ in the following example.

\begin{eg}
% The reason why \tr_{\wt_1} = s_0 s_2 is because
% s_0 s_2 A_\null = A_{s_0 s_2} and check this on the picture in Figure 1 and you will see this is A_\null + \Lambda_1
In type $\widetilde{\A}_2$, we will compute the conclusion of Lemma~\ref{lem:action} for $\eta = \Lambda_1$ and $p=1,2$.  
On the one hand, $\tr_{\wt_1}=s_0 s_2$ and, from Lemma~\ref{lem:pseudo_1}, we know that 
$A_{\tr_{\Lambda_1}}=\Anull+\Lambda_1$.  We can also compute that $\phi(z_{\Lambda_1})=s_1 s_0,$ 
$\phi^2(z_{\Lambda_1})=s_2 s_1$.  We observe from Figure~\ref{eg:walk} that
\[A_{\phi(z_{\Lambda_1})} = \Anull-\Lambda_1+\Lambda_2 
\text{ and } A_{\phi^2(z_{\Lambda_1})} = \Anull-\Lambda_2.\]
	
On the other hand, we compute that 
\[c \Lambda_1 = (s_1 s_2) \Lambda_1 = s_1 \Lambda_1 = \Lambda_0-\Lambda_1+\Lambda_2,\]
so that \[c^2 \Lambda_1 = c (\Lambda_0-\Lambda_1+\Lambda_2) = 2\Lambda_0-\Lambda_2.\]

Comparing the two sides, we conclude that $A_{\phi^p(z_{\Lambda_1})}=\Anull+c^{p}(\Lambda_1)$ for $p=1,2$.

	%$ has vertices $\Lambda_1, -\Lambda_0+2\Lambda_1,$ and $-\Lambda_0+\Lambda_1+\Lambda_2$  has vertices $\Lambda_0,\Lambda_0+\Lambda_1-\Lambda_2,$ and $2\Lambda_0-\Lambda_2$, so that $A_{\phi(z_{\Lambda_1})}	
%	$c \Lambda_1 = s_2 s_1 \Lambda_1 = 2\Lambda_0-\Lambda_2$
	
	%The action of $\phi$ on the alcoves is described in Lemma \ref{lem:action}.  The alcove
%$A_{s_1 s_0}$ has vertices $\{ \Lambda_1, -\Lambda_0+2 \Lambda_1, \Lambda_0+\Lambda_1 + \Lambda_2 \}$.
%Acting $c^{-1}$ on each of these vectors gives 
%$\{ - \Lambda_1 + \Lambda_2 + \Lambda_0, - \Lambda_1 + 2 \Lambda_2, \Lambda_2 \}$.

%\todo{
%this for some reason does not agree with the following calculation
%which follows by acting with $p=1$ of $c^{-1} v$ for $v \in
%\{ \Lambda_1, -\Lambda_0+2 \Lambda_1, \Lambda_0-\Lambda_1 + \Lambda_2 \}$)
%)
% $\{ \Lambda_0 - \Lambda_1 + \Lambda_2,
%-(\Lambda_0 - \Lambda_1 + \Lambda_1) + 2( \Lambda_0 - \Lambda_1 + \Lambda_2) 
%= \Lambda_0 - 2 \Lambda_1 + 2 \Lambda_2,
%(\Lambda_0 - \Lambda_1 + \Lambda_1) - ( \Lambda_0 - \Lambda_1 + \Lambda_2) + (\Lambda_0 - \Lambda_1 + \Lambda_0)
%=  \Lambda_0 - \Lambda_2 \}$

%I don't know why these coordinates don't agree with the calculation that I had before but
%I can't seem to do this calcualtion right now.}

\begin{center}
\raisebox{-.7in}{\includegraphics[width=1.5in]{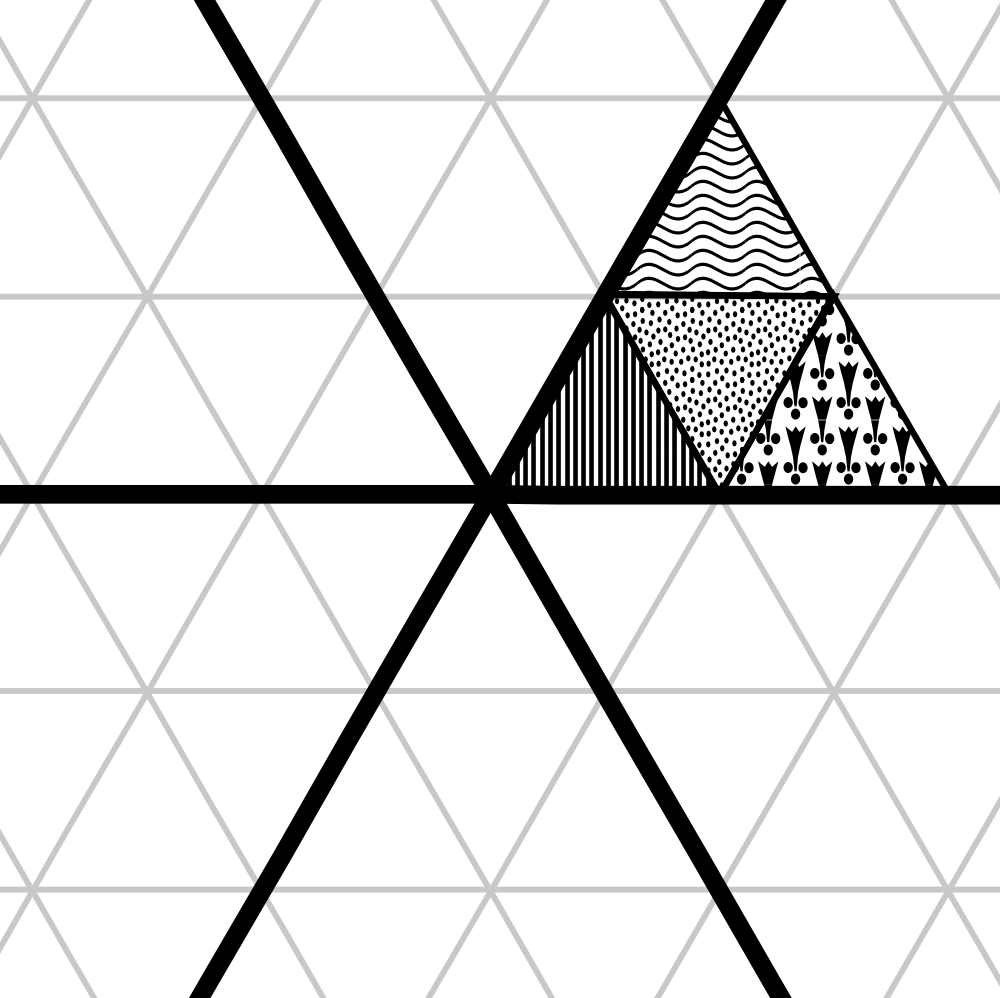}}
$\substack{\longrightarrow\\\phi}$
\raisebox{-.7in}{\includegraphics[width=1.5in]{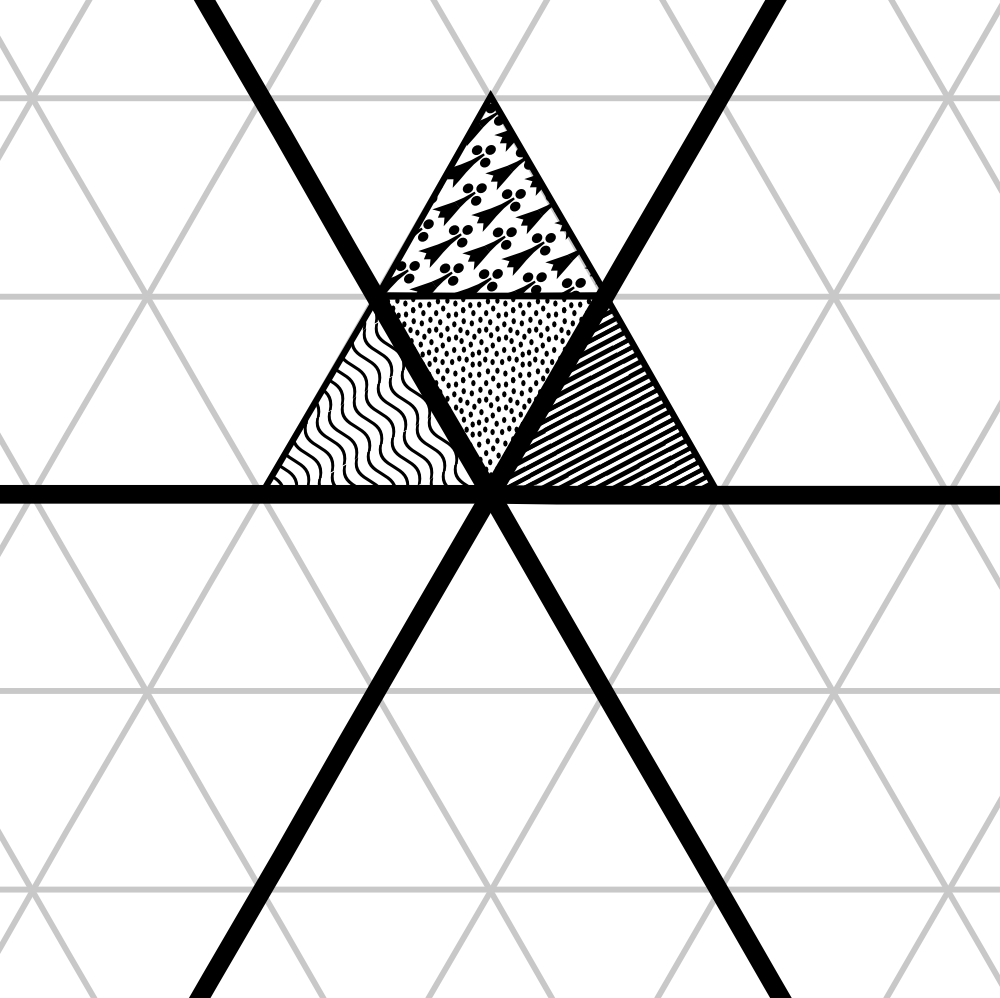}}\hskip .2in\includegraphics[width=.5in]{direction2.pdf}
\end{center}
\end{eg}

%\todo{
%Does this explain why $\phi( m \Anull ) + (m-1) \Lambda_1 = m \Anull$?
%}

%\todo{
%What happens if I act $c$ or $c^{-1}$ directly on the alcoves?  Probably $c A_w$ is wrong, but $A_w \rightarrow %A_{w c}$
%makes more sense.  

%Conjecture: $A_{w c}$ is $A_w$ rotated about the origin.  hmmm, why?

%This then has a picture in the $A_2$ case.  Lemma 5.2 explains where $m \Anull$
%is sent under the action of $c$.

%Conjecture: $A_{w+c} + m \Lambda_1$ (or $ + m \Lambda_k$) maps $m \Anull \rightarrow m \Anull$.
%}

We now draw on the established analogy to $k$-Schur functions.  
Combining Theorems~\ref{thm:bbtz} and~\ref{thm:LM}, we calculate:

\[
s^{(k)}_{\cup_{j=1}^{m-1} R_{i_j}} = \prod_{j=1}^{m-1} s^{(k)}_{R_{i_j}} 
= \prod_{j=1}^{m-1} \sum_{\Lambda \in \A_k \wt_i} \uu({\tr_\wt}) = \uu(\mathfrak{r}(\cup_{j=1}^{m-1} R_{i_j}))+\ldots.
%(\uu(\tr_{\wt_{i_j}})+\ldots)
\]

%obtained by 
%while the right-hand side consists of products of pseudo-translations which are themselves
%pseudo-translations by weights
The leftmost expression is indexed by the $k$-bounded partition $\cup_{j=1}^{m-1} R_{i_j}$.  By the bijection between affine Grassmannian elements and bounded partitions, a reduced word for the element $\uu(\mathfrak{r}(\cup_{j=1}^{m-1} R_{i_j}))$ may be obtained by filling the shape $\cup_{j=1}^{m-1} R_{i_j}$. This filling will contain certain shifts $\phi^d(\mathfrak{r}(R_{i_j}))$, so that the desired element appears on the right-hand side by Lemma~\ref{lem:contains_cycles}.

Just as Theorem~\ref{thm:bbtz} was the algebraic analogue 
of Proposition~\ref{prop:rectangles_and_fund_weights}, since we know that the weights 
$\sum_{j=1}^{m-1} \wt_{i_j}$ are in bijection with the maximal alcoves in $m\Anull$, 
this immediately suggests our main theorem, a characterization of the maximal elements of $Y_m^k$.

\begin{thm}
\label{thm:mainthm}
The maximal $k$-bounded partitions in $Y_m^k$ are the $R_{i_1,\ldots,i_{m-1}}:=\cup_{j=1}^{m-1} R_{i_j},$ so that \[Y_m^k = \{ \bndd \in \Pk : \bndd \subseteq R_{i_1,\ldots,i_j} \text{  for some } 1\leq i_1, \dots, i_{m-1} \leq k\}\].
\end{thm}

The proof of this theorem will be given in Section~\ref{sec:proofs}.  
By the geometric description of $Y^k_m$ in Definition~\ref{def:ymk}, the poset $Y^k_m$ 
has a $(k+1)$-fold cyclic symmetry.  This symmetry is explored further in Section~\ref{sec:app_sym}.  
Examples of these posets are given in Figures~\ref{eg:Y24} and~\ref{eg:Y33}. 

\begin{figure}
\begin{center}
\includegraphics[width=3in]{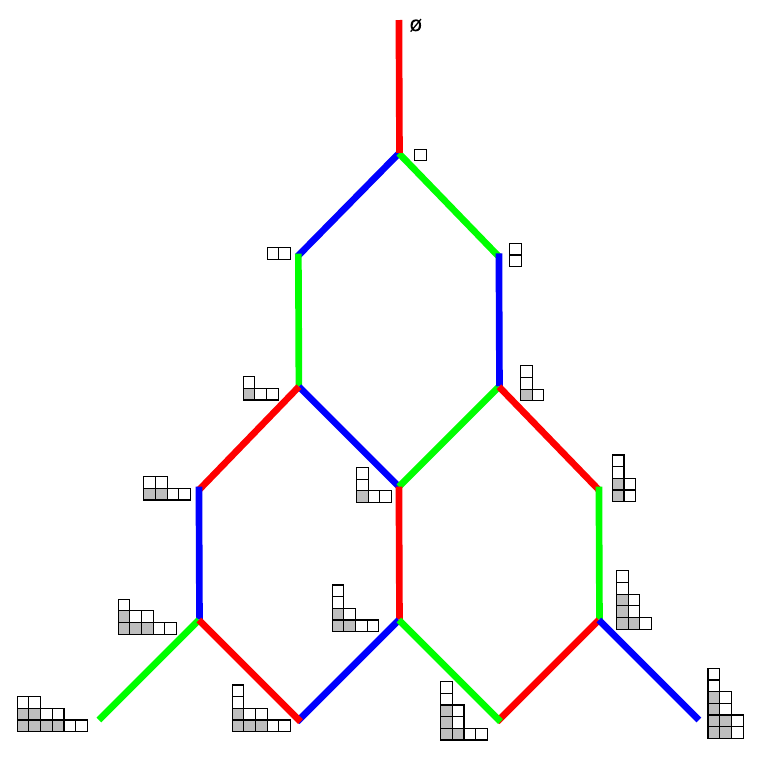}
\end{center}
\caption{The poset $Y^2_4$ labeled by $3$-cores.  The cores in this diagram are in bijection with the
$4$-fold dilation of the fundamental alcove of the hyperplane arrangement in Figure \ref{eg:walk}
or in Figure \ref{notfigure1}.
The cells in the partitions with hook-length 
greater than 3 are shaded.  The edge colors correspond to the content of the cells being added: 
red is content $0 \mod 3$, blue $1 \mod 3$, and green $2 \mod 3$.  
This poset exhibits a dihedral symmetry of order $3$.}
%A reflection in this symmetry
%is realized through conjugation of the $3$-cores.  }
\label{eg:Y24}
\end{figure}

\begin{figure}
\begin{center}
\includegraphics[scale=.7]{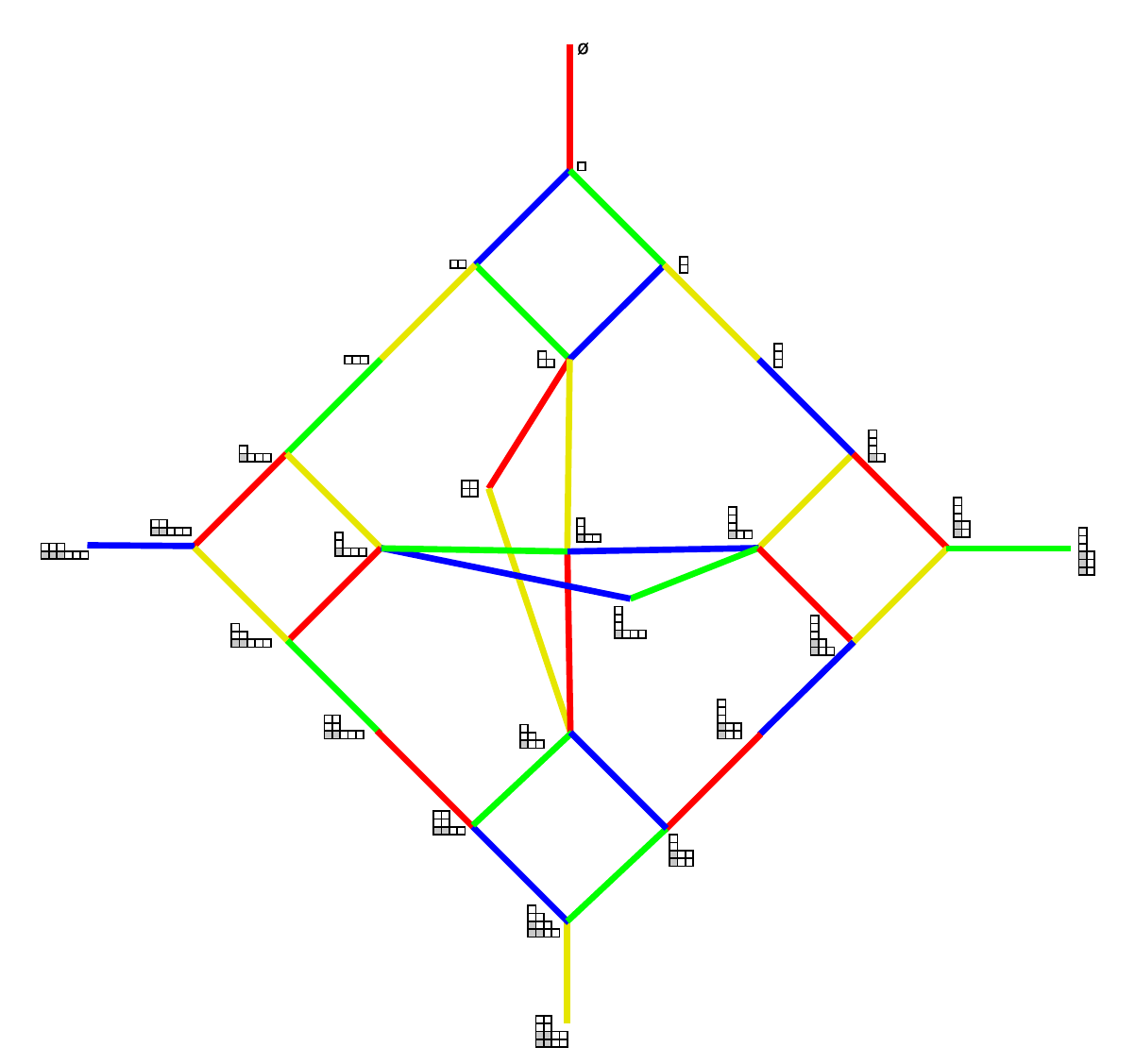}
\end{center}
\caption{The poset $Y^3_3$ labeled by $4$-cores.  The cells in the partitions with hook-length greater 
than 4 are shaded.  The edge colors correspond to the content of the cells being added: red is 
content $0 \mod 4$, blue $1 \mod 4$, yellow $2 \mod 4$, and green $3 \mod 4$.  
This poset exhibits a dihedral symmetry of order $4$.}%and 
%the corresponding 3-bounded partitions are obtained by deleting shaded cells and left justifying the partition.
\label{eg:Y33}
\end{figure}

%We now fix an integer $m >1$. With Theorem \ref{thm:LM} in mind, we will study all partitions contained in a product of $m$ rectangles. Let $Y^k_m$ denote the subposet of the $k$-Young's lattice which contains all partitions contained in a stack of $m-1$ of the $k$-rectangles (so $\lambda \in Y^k_m$ if $\lambda \subset R_{i_1}\cup R_{i_2}\cup \dots \cup R_{i_{m-1}}$ for some $i_1, \dots, i_{m-1}$).  By this definition,
%$Y^k_2 = Y^k$ from the beginning of Section \ref{sec:su}.

%We will prove this by appealing to the geometric description of Suter symmetry. The collection of alcoves in the dominant chamber which are bounded by the affine hyperplane $H_{\alpha_0, m}$ again inherits the cyclic $k+1$ symmetry of the fundamental alcove, thus proving that a cyclic $k+1$ symmetry exists on the alcoves. It remains to be shown that the alcoves in the dominant chamber bounded by the hyperplane $H_{\alpha_0,m}$ are in bijection with the partitions which are contained in a product of $m-1$ rectangles.
%Once we have shown this, our main theorem, that $Y_m^k$ has a cyclic $k+1$ action, will follow.

%%%%%%%%%%%%%%%%%%%%%%%%%%%%%%%%%%%%%%%%%%%%%%%%%%%%%%%%%%%%%%%%%%%%%%%%%%%%%%%%%%%%%%%%%%
\section{Alcoves with a facet on the hyperplane $H_{\alpha_0,m}$}
\label{sec:proofs}
%%%%%%%%%%%%%%%%%%%%%%%%%%%%%%%%%%%%%%%%%%%%%%%%%%%%%%%%%%%%%%%%%%%%%%%%%%%%%%%%%%%%%%%%%%

Recall that for two partitions $\bndd_1, \bndd_2$, we let $\bndd_1 \cup \bndd_2$ 
denote the partition obtained by combining the parts of $\bndd_1$ and $\bndd_2$ and 
placing them into non-increasing order.  

\begin{defn}
Let $R_{i_1,\ldots,i_{m-1}}$ be the $k$-bounded partition $R_{i_1} \cup \dots \cup R_{i_{m-1}}$ for $1 \leq i_j \leq k.$
\end{defn}

\begin{lemma}
\label{lem:numrects}
There are $\binom{m-1+k-1}{k-1}$ distinct $k$-bounded partitions $R_{i_1,\ldots,i_{m-1}}$.
\end{lemma}

\begin{proof}
%Each partition $R_{i_1,\ldots,i_{m-1}}$ is the union of some number of the partitions $R_i$.
The number of $R_{i_1,\ldots,i_{m-1}}$ is the number of ways to pick $m-1$ objects from 
a set of $k$ objects with repetition. 
\end{proof}

%\begin{comment}

\begin{thm}
Let $i_1 \geq i_2 \geq \cdots, \geq i_{m-1}$.

\begin{itemize}
\item The $(k+1)$-core $\mathfrak{c}(R_{i_1,\ldots,i_{m-1}})$ is the partition 
\[ \lambda:=\left( (\sum_{j=1}^{m-1} i_j)^{k+1-i_1}, (\sum_{j=2}^{m-1} i_j)^{k+1-i_2}, 
\ldots, (i_{m-1})^{k+1-i_{m-1}}\right) .\]

\item The affine Grassmannian permutation $\mathfrak{r}(R_{i_1,\ldots,i_{m-1}})$ 
is the pseudo-translation $\tr_{\sum_{j=1}^{m-1} \wt_{i_j}}$.
\end{itemize}
\label{thm:characterizations}
\end{thm}

\begin{proof}
For the first statement, it is easy to check that $\lambda$ is the $(k+1)$-core resulting 
from $R_{i_1,\ldots,i_{m-1}}$ by applying the bijection $\mathfrak{c}$ using the combinatorial
description of this map found in \cite{LM2} or \cite{LLMSSZ}.  We provide a proof by example for the sequence
of rectangles $R_{4,4,3,1}$ with $k=4$ in Example \ref{eg:pushout}.

To prove the second statement, recall that the element $\mathfrak{r}(R_{i_1,\ldots,i_{m-1}})$
is a reading word of the filling of $R_{i_1,\ldots,i_{m-1}}$ described in Theorem \ref{thm:bounded_to_grassmannian}.
We factor the reduced word  
into the piece in the rectangle $R_{i_1,\ldots,i_{m-2}}$ and the piece in $R_{i_{m-1}}$.  Let $d=\sum_{j=1}^{m-2} i_j$, so that
\[
\mathfrak{r}(R_{i_1,\ldots,i_{m-1}}) 
= \mathfrak{r}(R_{i_1,\ldots,i_{m-2}}) \phi^{d}(\mathfrak{r}(R_{i_{m-1}})) 
\]
%Therefore if we define the map which acts on reduced words $\phi^a( s_i ) = s_{i+a \mod (k+1)}$, 

By induction, we assume that 
$\mathfrak{r}(R_{i_1,\ldots,i_{m-2}}) \Anull = A_{R_{i_1,\ldots,i_{m-2}}} = \Anull + \sum_{j=1}^{m-2} \wt_{i_j}$.
Proposition~\ref{prop:oneone} says that the label of the translation of the origin in 
$\Anull + \sum_{j=1}^{m-2} \wt_{i_j}$ is
$\lbl( \sum_{j=1}^{m-2} \wt_{i_j} ) = d$, while the labels of the translate of the fundamental weights 
$\wt_i$ are $\lbl( \wt_i + \sum_{j=1}^{m-2} \wt_{i_j} ) = i+d$.  Thus, reflecting in the hyperplane 
opposite from $\wt_i$ in $\Anull$ and then translating is equivalent to translating and then reflecting in the hyperplane 
opposite from 
$\phi^d(\wt_i)$: 
\[A_{s_i}+\sum_{j=1}^{m-2} \wt_{i_j}= A_{\mathfrak{r}(R_{i_1,\ldots,i_{m-2}}) \phi^d(s_i)}\]
and, more generally, for any $u \in \widetilde{\A}_k$ we have 
\[A_{u}+\sum_{j=1}^{m-2} \wt_{i_j}= A_{\mathfrak{r}(R_{i_1,\ldots,i_{m-2}}) \phi^d(u)} .\]

By Lemma~\ref{lem:pseudo_1} we know that $\mathfrak{r}(R_{i_{m-1}}) \Anull = \Anull + \wt_{i_{m-1}}$, so that 
\[\mathfrak{r}(R_{i_1,\ldots,i_{m-1}}) \Anull  
= A_{\mathfrak{r}(R_{i_1,\ldots,i_{m-2}}) \phi^{d}(\mathfrak{r}(R_{i_{m-1}}))}  = 
A_{\mathfrak{r}(R_{i_{m-1}})} +\sum_{j=1}^{m-2} \wt_{i_j} = \Anull + \sum_{j=1}^{m-1} \wt_{i_j},\] 
so that $\mathfrak{r}(R_{i_1,\ldots,i_{m-1}}) = \tr_{\sum_{j=1}^{m-1} \wt_{i_j}}$.
%is $z_{\Lambda_{i_{m-1}}}$.
%By induction we can assume that $\mathfrak{r}(R_{i_2,\ldots,i_{m-1}})$
%is $z_{\sum_{j=2}^{m-1} \Lambda_{i_j}}$

%, therefore
%$\phi^{i_1}(\mathfrak{r}(R_{i_2,\ldots,i_{m-1}}))$ is the pseudo-translation
%$z_{\sum_{j=2}^{m-1} c^{i_1} \Lambda_{i_j}}$ by Lemma \ref{lem:contains_cycles}.

%Therefore
%\[
%\mathfrak{r}(R_{i_1,\ldots,i_{m-1}}) \Anull = z_{\Lambda_{i_1}}
%z_{\sum_{j=2}^{m-1} c^{i_1} \Lambda_{i_j}} \Anull = 
%z_{\Lambda_{i_1}}( \Anull + \sum_{j=2}^{m-1} c^{i_1} \Lambda_{i_j} )\]

%\todo{ the conclusion needs to be that this is $\Anull + \sum_{j=1}^{m-1} \Lambda_{i_j}$ }
\end{proof}

\begin{prop}
Let $w = \mathfrak{r}(R_{i_1,\ldots,i_{m-1}})$ and let $A_w$ have vertices $v_0,v_1,\ldots,v_k$ 
with $L(v_j)=j$.  Then some vertex $v_i$ of $A_{w}$ lies on the hyperplane $H_{\alpha_0,m-1}$ 
and the facet defined by the vertices $\{v_j : j \neq i \}$ lies on the hyperplane $H_{\alpha_0,m}$.  
Furthermore, $s_i$ is the unique right ascent of $w$ such that $w s_i$ is Grassmannian.
%$$The facet defined by ${v_j : j \neq i}$ is the unique ascent of $w$
\label{prop:one_ascent}
\end{prop}

%\todo{ what does the label have to do with the proposition?} Ans: it serves to make the statement more
%precise even if we don't use the labeling in the proof.

\begin{proof}
The facet of $\Anull$ defined by $\{\Lambda_i : 1\leq i \leq k\}$ lies on the hyperplane 
$H_{\alpha_0,1}$.  Now $A_w = \Anull+\sum_{j=1}^{m-1} \wt_{i_j}$, and the vertex $v_i$ which 
is a translate of the origin has coordinates $\sum_{j=1}^{m-1} \wt_{i_j}$.  
Since \[ \left\langle \sum_{j=1}^{m-1} \wt_{i_j}, \alpha_0 \right\rangle = m-1, \]
we see that $v_i$ lies on $H_{\alpha_0,m-1}$.  The vertices of $A_w$ which are not 
translates of the origin will have weight  $v_d = \wt_d + \sum_{j=1}^{m-1} \wt_{i_j}$ for $1 \leq d \leq k$.  
They therefore satisfy
\[\left< v_d, \alpha_0 \right> = 1 + \sum_{j=1}^{m-1} \langle \wt_{i_j},\alpha_0 \rangle = m, \] 
and so lie on the wall $H_{\alpha_0,m}$.

To show that $s_i$ is the unique ascent of $w$ such that $w s_i$ is Grassmannian, we use the 
corresponding $(k+1)$-core.  We claim that the only addable residue is 
$i = \sum_{j=1}^{m-1} i_j \mod k+1$. From Theorem~\ref{thm:characterizations}, the core 
$\mathfrak{r}(R_{i_1,\ldots,i_{m-1}})$ consists of the rectangles $R_{i_j}$ arranged in skew fashion.  
Cells which are on the opposite sides of a rectangle $R_{i_j}$ will have have the same residue 
because they are separated by a hook of length $k+1$. Therefore only one residue is addable.  
Finally, note that the length of the first row of $\mathfrak{c}(R_{i_1,\ldots,i_{m-1}})$ is $i=\sum_{j=1}^{m-1} i_j$, so that $i$ is indeed the addable residue.

% is 
%	obtained by appending rectangles ordered by their widths
%	in a skew fashion, stacking the rectangles so 
%	that adjacent rectangles share neither row nor column.  
%The fundamental weights $\Lambda_i$ all satisfy $\langle \Lambda_i, \alpha_0 \rangle = 1$
%and are the coordinates of the vertices of this facet. 
%Since $A$ is a translate of the fundamental alcove,
%$A = \Anull + (\Lambda_{i_1} + \dots + \Lambda_{i_{m-1}})$.
%The vertex which is a translate of the origin has coordinates
%$\Lambda_{i_1} + \dots + \Lambda_{i_{m-1}}$ and since
%\[ \left< \Lambda_{i_1} + \dots + \Lambda_{i_{m-1}}, \alpha_0 \right> = m-1 \]
%it lies on $H_{\alpha_0,m-1}$.
%The vertices of
%$A$ which are not translates of the origin will have weight
%$v_d = \Lambda_d + \Lambda_{i_1} + \dots + \Lambda_{i_{m-1}}$.  They will satisfy
%\[\left< v_d, \alpha_0 \right> = 1 + \sum_{j=1}^{m-1} \langle \Lambda_{i_j},\alpha_0 \rangle = m, \] 
%and so will lie on the wall $H_{\alpha_0,m}$.
\end{proof}

%\begin{lemma}\label{lemma:uniqueresidue}
%Let $\mathcal{R} = \mathfrak{c}(R_{i_1} \cup \dots \cup R_{i_{m-1}})$ be a $(k+1)$-core whose $k$-bounded partition
%is a union of a sequence of rectangles, 
%then $\mathcal{R}$ has only one addable residue. That is there exists a unique $i$ for which $u_i \mathcal{R} \neq 0$.
%\end{lemma}
%\todo{I am not sure if this proof is clear.  I believe it is true that
%``The core $\mathcal{R}$ is 
%obtained by appending rectangles ordered by their widths
%in a skew fashion'' but we certainly haven't shown this.}
%\begin{proof}
	
%	The only residue which is addable is $i = i_1 + \dots + i_{m-1}$. The core $\mathcal{R}$ is 
%	obtained by appending rectangles ordered by their widths
%	in a skew fashion, stacking the rectangles so 
%	that adjacent rectangles share neither row nor column.  
%	Cells which are on the opposite sides of a rectangle in the core will have have the
%	same residue because they are separated by a hook of length $k$. Therefore only one residue is addable.  
%	The length of the first row of $\mathcal{R}$ will be $i = i_1 + \dots + i_{m-1}$
%	and so it is also the residue of the addable corner.
%\end{proof}

\begin{lemma}
Let $u$ be an affine Grassmannian permutation, and let $\wt = \sum_{j=1}^{m-1} \wt_{i_j}$ 
be a vertex of the alcove $A_u$.  If $w=\mathfrak{r}(R_{i_1, \ldots, i_{m-1}})$, then $u \leq w$.
\label{lem:less_than_my_weight}
\end{lemma}

\begin{proof}
	%, since crossing the hyperplane will increase the length of the corresponding core and we know there is only one reflection which will add a cell to $R$, by Lemma \ref{lemma:uniqueresidue}.
	%By Lemma \ref{prop:one_ascent}, the core $\core$ corresponding to $w$ has a unique addable residue $r$, which corresponds to crossing the hyperplane $H_{\alpha_0,m}$.  Applications of all other generators $s_i$ for $i \neq r$ must therefore either decrease the size of $\core$, or cause $w$ to cease being Grassmannian.  
	%Since $w$ is a Grassmannian permutation with a unique ascent $s_i$ such that $s_i w$ is Grassmannian, and since this $s_i w$ no longer contains $\wt$, if $\ell(u)-\ell(w)=0$ then $w$ and $u$ must be equal.%We argue by induction on $\ell(u)-\ell(w)$.  
	
	 %containing the vertex $\wt$.
	
Let ${\mathcal C}$ be the set of hyperplanes $H_{\alpha,p}$ 
that contain the weight $\Lambda$ and that have $p>0$.  Since $A_w$ has a unique 
(Grassmannian) ascent in the facet not containing $\Lambda$, a reflection across any of the hyperplanes
containing $\Lambda$ must decrease the length of $w$
and therefore the inversion set of $w$ 
contains all hyperplanes in ${\mathcal C}$.
%In fact, any reflection that leaves $\Lambda$ 
%fixed and takes an alcove in the dominant chamber to another is a reflection in one of the hyperplanes in $\mathcal C$.  
Since $A_u$ has $\Lambda$ as a vertex and since $A_u$ is Grassmannian, any reflection of $A_u$
across a hyperplane in ${\mathcal C}$ will remain Grassmannian and some subset of these reflections will send
$A_u$ to $A_w$.  
Therefore $\inv(w) = \inv(u) \cup {\mathcal C}$ and, in particular, the
inversion set of $u$ is contained in the inversion set of $w$, so that $u\leq w$ in weak order.
%	contain $\Lambda$, any alcove that lies on a minimal length gallery between $A_w$ and $A_u$ will also contain $\Lambda$, and so a subset of these hyperplanes separate $A_w$ from $A_u$.  
	% there is an element $s_{a_1} s_{a_2} \cdots s_{a_x}$ which takes $A_w$ to $A_u$.
	%(i.e. $ A_{s_{a_1} s_{a_2} \cdots s_{a_x} w} = B$ for some $a_j \neq r$). 
	%Therefore $\lambda \subset R$, since $\lambda = s_{a_x} \cdots s_{a_1} R$.
	%	(b) If w is a grassmannian permutation such that the weight \lambda is one of the vertices of the alcove corresponding to w, then w \leq u(\lambda).
%Unless w contains the weight \lambda = w^{-1} \lambda_i and w's unique grassmannian cover is such that \lambda_i is the weight that does not lie on the hyperplane H between w and its grassmannian covers, we can choose some other hyperplane to cross.  This preserves the property that s_i w contains \lambda.
\end{proof}

We now prove Theorem~\ref{thm:mainthm}, characterizing the maximal $k$ bounded partitions in $Y_m^k$.

%\begin{thm}\label{cor:all_below}
%Let $\bndd$ be a $k$-bounded partition and suppose that the alcove 
%$A_{\bndd}$ is in $m\Anull$.  Then 
%there exists an $R = R_{i_1, \ldots, i_{m-1}}$ such that $\bndd \subseteq R$.
%\end{thm}
\begin{proof} 

The proof is by induction on $m$. When $m = 1$, the statement is trivial.
%; the only dominant 
%alcove bounded by $H_{\alpha_0,1}$ is the fundamental alcove, which corresponds to the empty 
%partition $\emptyset$, which is contained in an empty product of rectangles.

Now fix $m>1$. If $A_\bndd$ is in $(m-1)\Anull$, then the statement follows by induction, since
$\bndd \subseteq R_{i_1,\ldots,i_{m-2}} \subseteq R_{i_1,\ldots,i_{m-2},i_{m-1}}$ for any 
other $1 \leq i_{m-1} \leq k$.  We may therefore assume that $A_\bndd$ is between $H_{\alpha_0,m-1}$ 
and $H_{\alpha_0,m}$ and so has at least one vertex on $H_{\alpha_0,m-1}$.  The result now 
follows from Lemma~\ref{lem:less_than_my_weight}.% and Theorem~\ref{thm:characterizations}.
%Let that vertex be $\wt = \sum_{j=1}^{m-1} \wt_{i_j}$.  We claim that $\lambda \subseteq R_{i_1,\ldots,i_{m-1}}$. 
\end{proof}

Theorem~\ref{thm:mainthm} validates our explicit description of $Y_m^k$.  In particular, 
we can translate the geometric properties of $m\Anull$ to combinatorial statements on $Y_m^k.$

% the following results.  One is that
%we know how many elements in $Y_m^k$ there are.  The other states that there is a
%cyclic action on this poset.

\begin{prop} \label{prop:volume}
The number of partitions in $Y_m^k$ is $m^k$.
\end{prop}

\begin{proof}
Since $m\Anull$ is an $m$-fold dilation in $k$-dimensional space, 
$\mathrm{vol}(m\Anull) = m^k \mathrm{vol}(\Anull).$  Because $Y_m^k$ 
is in bijection with the alcoves in $m\Anull$, there must be $m^k$ of them.
%which is a dilation by $m$ on a side of a single
%fundamental alcove, the volume will be
%$m^k$ times the original, hence 
%and an equal number of elements in $Y_m^k$.
\end{proof}

\begin{comment}
\todone{ fill in some details here.  Theorerm  \ref{mainthm} slightly more than just a corollary to 
%\ref{cor:all_below}
 
%{\bf CB:} Added short proof. What else did you want to include?

Here the cyclic action isn't explicit and this is what the reviewer
wanted to know how to calculate this.  Instead of simply just stating
there is a cyclic action, I want to state what the cyclic action is and
how to compute it.

\begin{thm}\label{mainthm}
The set $Y^k_m$ has a cyclic $k+1$ action.
\end{thm}

\begin{proof}
By  Corollary  \ref{cor:all_below},  $Y_m^k$  corresponds  precisely  with  alcoves  in  
$m  \Anull$.      The  region  in  the  dominant  chamber  bounded
by $H_{\alpha_0,m}$ has the same shape as the fundamental alcove; the lengths of the edges 
of the fundamental alcove have been multiplied by $m$ in $H_{\alpha_0,m}$. Since the 
fundamental alcove has a cyclic $k+1$ action which is inherited from the affine
Dynkin diagram, the collection of alcoves in this region inherits the 
cyclic $k+1$ action.
% from the fundamental alcove,
%\todo{why?  It is the same shape as the fundamental alcove
%because it is the region bounded by $H_i$ for $1 \leq i \leq k$ and $H_{\alpha_0,m-1}$
%which is the fundamental alcove multiplied by $m$.
%
%{\bf CB:} Correct. And the vertices of the fundamental alcove are indexed by dominant weights, which correspond to the nodes of the dynkin diagram. 
%
%}
%or equivalently, the affine Dynkin diagram.
\end{proof}

}
\end{comment}

%%%%%%%%%%%%%%%%%%%%%%%%%%%%%%%%%%%%%%%%%%%%%%%%%%%%%%%%%%%%%%%%%%%%%%%%%%%%%%%%%%%%%%%%%%
\section{Applications}
\label{sec:applications}
%%%%%%%%%%%%%%%%%%%%%%%%%%%%%%%%%%%%%%%%%%%%%%%%%%%%%%%%%%%%%%%%%%%%%%%%%%%%%%%%%%%%%%%%%%

We study applications of our results.

%%%%%%%%%%%%%%%%%%%%%%%%%%%%%%%%%%%%%%%%%%%%%%%%%%%%%%%%%%%%%%%%%%%%%%%%%%%%%%%%%%%%%%%%%%
\subsection{From alcoves to $k$-bounded partitions or $(k+1)$-cores}
\label{sec:app_alcoves_to_kbounded}
%%%%%%%%%%%%%%%%%%%%%%%%%%%%%%%%%%%%%%%%%%%%%%%%%%%%%%%%%%%%%%%%%%%%%%%%%%%%%%%%%%%%%%%%%%

%We have explained the procedure to go from elements of $W$ to alcoves, and in particular, elements of $W^0$ to alcoves in $C$.
%We have also explained the bijection between $k$-bounded partitions and elements of $W^0$ and (through the action on $k+1$-cores)
%from reduced words of $W^0$ to $k+1$-cores.  Furthermore, the bijection from $k+1$-cores to $k$-bounded partitions
%is given in Equation \eqref{eq:bijkp1core2kpart}.  

Recall  that  the  bijections  for  passing  between  $k$-bounded  partitions  or  $(k+1)$-cores  
$\core$  and  affine  Grassmannian  permutations  $w$  was  a  several  step  process  that  all  
relied  on  producing  \textit{reduced  words}  for  $w$.    In  this  section,  we  give  a  
direct  bijection  for  going  from  the  vertices  of  an  alcove  to  its  corresponding  
$k$-bounded  partition  or  $(k+1)$-core.

\begin{thm}  Let $A_w$ be an alcove with vertices $v_d = \sum_{j=1}^k \wt_{i_j}$ for $0 \leq d \leq k$, 
and let $w_d = \mathfrak{r}(R_{i_1,\ldots,i_{k}})$.  The element $w$ is the weak order greatest 
common divisor of the $w_d$.
\label{thm:gcdw}
\end{thm}

\begin{proof} By Lemma~\ref{lem:less_than_my_weight}, we know that $w \leq w_d$, so that 
$w \leq \wedge_d w_d$ and $w$ is a lower bound for the $w_d$.  By Proposition~\ref{prop:one_ascent}, 
we see that for each Grassmanian ascent $s_j$ of $w$, there is some $w_d$ with $w s_j \not \leq w_d$, 
so that $w s_j \not \leq \wedge_d w_d$.  Therefore, each cover of $w$ is not a lower bound for 
$\wedge_d w_d,$ from which we conclude that $w = \wedge_d w_d$.% (that is, $w$ is the greatest lower bound for the $u_i$).
\end{proof}

We can therefore compute $w$ as the intersection of the inversion sets of $w_0,\ldots,w_k.$
Theorem~\ref{thm:gcdw} can be modified to allow us to also directly compute the corresponding
$(k+1)$-core $\mathfrak{c}(w)$ or the $k$-bounded partition
$\mathfrak{p}(w)$ from the coordinates of $A_w$.  The calculation involves taking intersections
of partitions where if $\ell = \text{min}(\ell(\mu), \ell(\gamma))$ then
$\mu \cap \gamma = (\text{min}(\mu_1,\gamma_1), \text{min}(\mu_2,\gamma_2), \ldots, \text{min}(\mu_\ell,\gamma_\ell))$.

%\begin{prop}
%\end{prop}
\begin{cor} \label{cor:intersect}
	Let $A_w$ be an alcove with vertices $v_d = \sum_{j=1}^k \wt_{i_j}$ for $0 \leq d \leq k$, 
	let $\bndd_d = R_{i_1,\ldots,i_{k}}$, $\core_d = \mathfrak{c}(\bndd_d)$, and let 
	$w_d = \mathfrak{r}(R_{i_1,\ldots,i_{k}})$.
	%Let ${\mathcal R}_d$ be the partition which is the union of $\eta^{(d)}_i$ copies of the rectangle $R_i$.
\begin{itemize}
\item If $\bndd$ is the $k$-bounded partition with cells that are in the intersection of all 
of the $\bndd_d$, then  $\bndd = \mathfrak{p}(w)$.

\item If $\core$ is the $(k+1)$-core with cells which are in the intersection of all of the 
$\core_d$, then $\core = \mathfrak{c}(w)$.
\end{itemize}
%	The intersection of the $(n+1)$-cores corresponding to $u_0,\ldots,u_n$ is the $(n+1)$-core corresponding to $w$.
\end{cor}

\begin{proof}
We prove the two items simultaneously.  By Lemma~\ref{lem:less_than_my_weight}, $w$ is less than 
each $w_d$ in weak order, for $0 \leq d \leq k$.  Therefore, there is a path from $w$ to each 
$w_d$---so that $\mathfrak{p}(w)$ and $\mathfrak{c}(w)$ are contained within the intersection 
of the diagrams of $\bndd_d$ and $\core_d$, respectively.  If there is a cell outside the 
diagram of $\mathfrak{p}(w)$ (resp. $\mathfrak{c}(w)$) that is in the intersection of the 
$\bndd_d$ (resp. $\core_d$), then there is such a cell that is addable from $\mathfrak{p}(w)$ 
(resp. $\mathfrak{c}(w)$).   This cell corresponds to an inversion that is not in $\inv(w)$, 
but that is in each of $\inv(w_d)$, contradicting Theorem~\ref{thm:gcdw}.	
\end{proof}

\begin{eg}
Let $w = s_0 s_1 s_2 s_1 s_0$.  The alcove $A_{w}$---the alcove directly before 
the end of the red gallery in Figure~\ref{eg:walk}---has vertices
$v_0 = 2 \wt_1 + 2 \wt_2$, $v_1 = 2 \wt_1 + \wt_2$ and $v_2 = \wt_1 + 2 \wt_2$.

Then $\bndd_0 = (2,2,1,1,1,1)$, $\bndd_1 = (2,1,1,1,1)$ and $\bndd_2 = (2,2,1,1)$.  
Intersecting these, we obtain $\bndd = (2,1,1,1) = \mathfrak{p}(w)$ as we depict 
graphically in the following Ferrer's diagrams for the partitions.

\newdimen\squaresize \squaresize=8pt
\newdimen\thickness \thickness=0.4pt

\begin{equation*}
\young{\cr\cr\cr\cr&\cr&\cr} \cap \young{\cr\cr\cr\cr&\cr} \cap \young{\cr\cr&\cr&\cr} = \young{\cr\cr\cr&\cr}
\end{equation*}

Similarly, $\core_0 = (6,4,2,2,1,1)$, $\core_1 = (4,2,2,1,1)$ and $\core_2 = (5,3,1,1)$.  
Intersecting these, we obtain the $3$-core $\core = (4,2,1,1) = \mathfrak{c}(w)$.	 
The diagrams for the partitions are pictured below with
the cells with skew length greater than 2 shaded.

\newdimen\squaresize \squaresize=8pt
\newdimen\thickness \thickness=0.4pt

\begin{equation*}
\young{\cr\cr\gsqr&\cr\gsqr&\cr\gsqr&\gsqr&&\cr\gsqr&\gsqr&\gsqr&\gsqr&&\cr} 
\cap \young{\cr\cr\gsqr&\cr\gsqr&\cr\gsqr&\gsqr&&\cr} 
\cap \young{\cr\cr\gsqr&&\cr\gsqr&\gsqr&\gsqr&&\cr} 
= \young{\cr\cr\gsqr&\cr\gsqr&\gsqr&&\cr}
\end{equation*}

%	The alcove $A_{s_1 s_0 s_1 s_2 s_1 s_0}$ considered in Figure~\ref{eg:walk} has vertices
%$v_0 = 2 \wt_1 + 2 \wt_2$, $v_1 = \wt_1 + 3 \wt_2$ and $v_2 = \wt_1 + 2 \wt_2$.
%Then ${\mathcal R}_0 = (6,4,2,2,1,1)$, ${\mathcal R}_1 = (7,5,3,1,1)$ and ${\mathcal R}_2 = (5,3,1,1)$.  Intersecting these, we obtain
%$\core = (5,3,1,1)$.
%Then 
\end{eg}

%%%%%%%%%%%%%%%%%%%%%%%%%%%%%%%%%%%%%%%%%%%%%%%%%%%%%%%%%%%%%%%%%%%%%%%%%%%%%%%%%%%%%%%%%%
\subsection{A cyclic action on $Y_m^k$}
\label{sec:app_sym}
%%%%%%%%%%%%%%%%%%%%%%%%%%%%%%%%%%%%%%%%%%%%%%%%%%%%%%%%%%%%%%%%%%%%%%%%%%%%%%%%%%%%%%%%%%

Fix $\widetilde{\A}_k$.  By Theorem~\ref{thm:mainthm}, the set of $k$-bounded partitions less than the $R_{i_1,\ldots,i_{m-1}}$
inherits the cyclic action on $m\Anull$.  In this section, we give several explicit descriptions of this action, which we denote $C^k_{m}: m\Anull \rightarrow m\Anull.$

\subsubsection{Apply $\phi$ and translate.}
%%%%%%%%%%%%%%%%%%%%%%%%%%%%%%%%%%%%%%%%%%%%%%%%%%%%%%%%%%%%%%%%%%%%%%%%%%%%%%%%%%%%%%%%%%

For $w = {\mathfrak r}(\lambda)$ with $\lambda \in Y_m^k$, let $C^k_{m}( A_w ) = A_{\phi^{-1}(w)} + (m-1) \Lambda_1$.  In fact, we may define the action of $C^k_{m}$ on any
vector $\gamma =  \sum_{i=1}^k \gamma_i \Lambda_i \in m\Anull$ as 
\begin{equation}\label{eq:imageoflambdai}
C^k_{m}( \gamma ) = (m - \sum_{i=1}^k \gamma_i) \Lambda_1 + \sum_{i=1}^{k-1} \gamma_i \Lambda_{i+1}~.
\end{equation}

%By Equation \eqref{eq:Amkdef}, i
We  can  characterize the weights in $m\Anull$  as the the set
	\begin{equation}\label{eq:Amkdef}
		\left\{  \sum_{i=1}^k  \lambda_i  \Lambda_i  :  \eta_i  \geq  0\hbox{  and  }
\sum_{i=1}^k  \lambda_i  \leq  m  \right\}.
	\end{equation}

%It is easy to check that if $\gamma \in m\Anull$, then $C^k_{m}(\gamma) \in m\Anull$.

%By  definition,  if  $\gamma_i  \geq  0$  and  $\sum_{i=1}^k  \gamma_i  \leq  m$  then  $\gamma  \in  m\Anull$.
Since $(m - \sum_{i=1}^k \lambda_i) \geq 0$, and the sum of the coefficients of $C^k_{m}(\gamma)$
is equal to $m - \gamma_k \leq m$, we see that $C^k_{m}( A_w )$ is also in $m\Anull$.  This transformation of $m\Anull$ is continuous since 
$$C^k_{m}( \gamma+\epsilon\tau) - C^k_{m}(\gamma) = -\epsilon \left(\sum_{i=1}^k \tau_i\right) \Lambda_1 
+ \sum_{i=1}^{k-1} \epsilon \tau_i \Lambda_{i+1}$$ hence adjacent alcoves are sent to adjacent alcoves.
%Also we clearly have that weights of the lattice are sent to weights.

It is reasonable to expect that we can realize the cyclic action using a Coxeter element $c$ and a translation.  
To this end, let $C_m^k(A_w) = c (A_w) +m\Lambda_k.$  In Lemma~\ref{lem:contains_cycles}, 
we computed that $c^{-p}\wt_i = \wt_0-\wt_p+\wt_{i+p}.$  Combining this with Equation~\eqref{eq:Amkdef}, 
it is easy to check that if $\eta \in m \Anull$, then $C_m^k(\eta) \in m \Anull$, 
and that this definition agrees with the one previously given.

We can check that for $1 \leq r \leq k$ that
$$\left(C^k_{m} \right)^r( \gamma ) = (m - \sum_{i=1}^k \gamma_i) \Lambda_r + \sum_{i=1}^{k} \gamma_i \Lambda_{i+r}$$
where the indices of $\Lambda_d$ are taken to be mod $k+1$ and $\Lambda_0 = 0$.  Moreover,
$$\left(C^k_{m} \right)^{k+1} (\gamma) = C^k_{m}\left( (m - \sum_{i=1}^k \gamma_i) \Lambda_k 
+ \sum_{i=1}^{k} \gamma_i \Lambda_{i+k} \right) = \gamma,$$
hence  $C^k_{m}$ has order $k+1$.

\begin{eg}
To give an example of this map we compute the orbit of a single alcove of $m \Anull$ when
$m=3$ and $k=3$.
Consider the $3$-bounded partition $\bndd = (3,1) \in Y_3^3$.  The corresponding 
alcove is indexed by the element ${\mathfrak r}((3,1)) = s_0 s_1 s_2 s_3$ is obtained from reading the tableau
\newdimen\squaresize \squaresize=12pt
\newdimen\thickness \thickness=0.4pt

\[\young{s_3\cr s_0& s_1& s_2\cr}.\]

The
bounding vectors of $s_0 s_1 s_2 s_3 \Anull $ are the images
of $\wt_0, \wt_1, \wt_2, \wt_3$ under the action of the element $s_0 s_1 s_2 s_3$.
In this case,
$s_0 s_1 s_2 s_3( \wt_0 ) =  \wt_1 +\wt_3 - \wt_0,$ %(1,0,0,-1)
$s_0 s_1 s_2 s_3( \wt_1 ) = \wt_2+\wt_3-\wt_0,$ %(1,1,0,-1)
$s_0 s_1 s_2 s_3( \wt_2 ) = 2 \wt_3-\wt_0,$ %(1,1,1,-1)
$s_0 s_1 s_2 s_3( \wt_3 ) = \wt_1+2\wt_3 - 2 \wt_0.$ %(1,0,0,-2)

The images of these vectors after applying $C_{3}^3$ are
$$C^3_{3}( \wt_1 + \wt_3 - \wt_0) = \wt_1 + \wt_2 - \wt_0 %= s_0 s_3( \wt_3 )
\hbox{ then }\mu_1 = (2,2,1,1,1)$$
$$C^3_{3}( \wt_2+\wt_3 - \wt_0) = \wt_1 + \wt_3 - \wt_0 %= s_0  s_3( 0)
\hbox{ then } \mu_2 = (3,1,1,1)$$
$$C^3_{3}( 2\wt_3 - \wt_0) = \wt_1 %= s_0 s_3( \wt_1 ) 
\hbox{ then }\mu_3 = (1,1,1)$$
$$C^3_{3}( 2\wt_3+\wt_1 - 2\wt_0) = \wt_2 %= s_0  s_3( \wt_2 )
\hbox{ then } \mu_4 = (2,2)$$
Hence, since $(2,2,1,1) \cap (3,1,1,1) \cap (1,1,1) \cap (2,2) = (1,1)$ and 
$\mathfrak{r}((1,1)) = s_0 s_3$, then $C^3_{3}( A_{s_0 s_1 s_2 s_3} ) = A_{s_0 s_3}$.  

Now we compute the action $C^3_3$ a second time and determine
$$C^3_{3}( \wt_1 + \wt_2 - \wt_0 ) = \wt_1 + \wt_2 + \wt_3 - 2 \wt_0 \hbox{ then }\mu_1 = (3,2,2,1,1,1)$$
$$C^3_{3}( \wt_1 + \wt_3 - \wt_0 ) = \wt_1 + \wt_2 - \wt_0\hbox{ then }\mu_2 = (2,2,1,1,1)$$
$$C^3_{3}( \wt_1 ) = 2 \wt_1 + \wt_2 - 2 \wt_0\hbox{ then }\mu_3 = (2,2,1,1,1,1,1,1)$$
$$C^3_{3}( \wt_2 ) = 2 \wt_1 + \wt_3 - 2 \wt_0\hbox{ then }\mu_4 = (3,1,1,1,1,1,1)$$
Hence, since $(3,2,2,1,1,1) \cap (2,2,1,1,1) \cap (2,2,1,1,1,1,1,1) \cap (3,1,1,1,1,1,1) = (2,1,1,1,1)$ and
$\mathfrak{r}((2,1,1,1,1)) = s_0 s_1 s_3 s_2 s_1 s_0$ and hence $C^3_{3}( A_{s_0 s_3} ) = A_{s_0 s_1 s_3 s_2 s_1 s_0}$.

Since the map $C^3_{3}$ has order $4$, we compute again the image of the alcove
$$C^3_{3}( \wt_1 + \wt_2 + \wt_3 - 2\wt_0 ) = \wt_2 + \wt_3 - \wt_0\hbox{ then }\mu_1 = (3,2,2)$$
$$C^3_{3}( \wt_1 + \wt_2 - \wt_0) =  \wt_1 + \wt_2 + \wt_3 - 2 \wt_0\hbox{ then }\mu_2 = (3,2,2,1,1,1)$$
$$C^3_{3}( 2 \wt_1 + \wt_2 - 2 \wt_0) = 2 \wt_2 + \wt_3 - 2 \wt_0\hbox{ then }\mu_3 = (3,2,2,2,2)$$
$$C^3_{3}( 2 \wt_1 + \wt_3 - 2 \wt_0) = 2 \wt_2 - \wt_0\hbox{ then }\mu_4 = (2,2,2,2)$$
Hence, since $(3,2,2) \cap (3,2,2,1,1,1) \cap (3,2,2,2,2) \cap (2,2,2,2) = (2,2,2)$ and
$\mathfrak{r}((2,2,2)) = s_0 s_1 s_3 s_0 s_2 s_3$ and hence $C^3_{3}( A_{s_0 s_1 s_3 s_2 s_1 s_0} ) 
= A_{s_0 s_1 s_3 s_0 s_2 s_3}$.%s_0 s_1 s_3 s_0 s_2 s_3

Therefore we have that
\begin{equation*}
A_{s_0 s_1 s_2 s_3} ~~~~\substack{C^3_3\\\longrightarrow\\{}}~~~~ A_{s_0 s_3} 
~~~~\substack{C^3_3\\\longrightarrow\\{}}~~~~ A_{s_0 s_1 s_3 s_2 s_1 s_0}
~~~~\substack{C^3_3\\\longrightarrow\\{}}~~~~ A_{s_0 s_1 s_3 s_0 s_2 s_3}
~~~~\substack{C^3_3\\\curvearrowright\\{}}
\end{equation*}
The alcoves correspond to the $3$-bounded partitions
\newdimen\squaresize \squaresize=8pt
\newdimen\thickness \thickness=0.4pt

\begin{equation*}
\young{\cr&&\cr}
~~~~\substack{C^3_3\\\longrightarrow\\{}}~~~~ \young{\cr\cr}
~~~~\substack{C^3_3\\\longrightarrow\\{}}~~~~ \young{\cr\cr\cr\cr&\cr}
~~~~\substack{C^3_3\\\longrightarrow\\{}}~~~~ \young{&\cr&\cr&\cr}
~~~~\substack{C^3_3\\\curvearrowright\\{}}
\end{equation*}
or their images under $\mathfrak{c}$ to the $4$-cores (with cells that have hook greater than $4$ shaded)
are given by the diagrams
\begin{equation*}
\young{\cr\gsqr&&&\cr}
~~~~\substack{C^3_3\\\longrightarrow\\{}}~~~~ \young{\cr\cr}
~~~~\substack{C^3_3\\\longrightarrow\\{}}~~~~ \young{\cr\cr\cr\gsqr&\cr\gsqr&&\cr}
~~~~\substack{C^3_3\\\longrightarrow\\{}}~~~~ \young{&\cr&\cr\gsqr&\gsqr&&\cr}
~~~~\substack{C^3_3\\\curvearrowright\\{}}~.
\end{equation*}

This calculation is consistent with the diagram in Figure \ref{eg:Y33} where we can see the orbit
of the $4$-core $(4,1)$ consists of precisely these four elements.
\end{eg}

Finding an action directly on $k$-bounded
partitions or on $(k+1)$-cores seems possible, but appears to be
more complicated than beautiful.  Using the characterization of Theorem~\ref{thm:mainthm}, 
an equivariant map between the geometric picture presented here 
and a certain set of words (coming from the abacus model of the $(k+1)$-cores) under 
rotation was given in~\cite{TW}. 
 
%H.\ Thomas and the second author used the description on words in~\cite{TW} 
%to prove a cyclic sieving phenomenon.

%%%%%%%%%%%%%%%%%%%%%%%%%%%%%%%%%%%%%%%%%%%%%%%%%%%%%%%%%%%%%%%%%%%%%%%%%%%%%%%%%%%%%%%%%%
%\subsubsection{Apply $c$ and translate.}
%%%%%%%%%%%%%%%%%%%%%%%%%%%%%%%%%%%%%%%%%%%%%%%%%%%%%%%%%%%%%%%%%%%%%%%%%%%%%%%%%%%%%%%%%%

%know that $\phi^p(z_{\Lambda_i}) = z_{c^{-p} \Lambda_i}.$

%We therefore must simply compute that $c$ acts appropriately.

%%%%%%%%%%%%%%%%%%%%%%%%%%%%%%%%%%%%%%%%%%%%%%%%%%%%%%%%%%%%%%%%%%%%%%%%%%%%%%%%%%%%%%%%%%
\subsubsection{Translate and apply $\phi$.}
%%%%%%%%%%%%%%%%%%%%%%%%%%%%%%%%%%%%%%%%%%%%%%%%%%%%%%%%%%%%%%%%%%%%%%%%%%%%%%%%%%%%%%%%%%
\label{subsec:translate}

%It is natural to ask if it is possible to do away with the translations that the definitions above
%require.  We answer this question in the affirmative if $m$ is coprime to $k+1$.
In the case that $m$ is coprime to $k+1$ it is possible to translate the region $Y^k$ so that
the identity alcove is at the center of the region and then rotation is an application of the element $\phi$.
We do this by modifying a construction of E.~Sommers to describe the dilated chamber.
Such dilations have also occurred in the work of J.~Y.~Shi.
%when $m$ is coprime to $k+1$.%, and when $m$ is a multiple of $k+1$.
%%% is this also the case when the center of the dilated chamber is a weight?

Note that if $m$ and $k+1$ are not relatively prime then there will not be an alcove in the
region $Y_m^k$ which is fixed by rotation.

\begin{defn}
Let $m$ be coprime to $k+1$, and write $m = a (k+1) + b$ so that $b \equiv m \mod (k+1)$.
Let $\Phi_b:=\{\alpha_i+\alpha_{i+1}+\cdots+\alpha_{i+b-1} : 1 \leq i \leq k-b+1\}$ be the set of roots of height $b$.
Define $D_m^{k}$ to be the region bounded by the hyperplanes 
\[H_m:= \{ H_{\alpha,-a} : \alpha \in \Phi_{b}\} \cup \{ H_{\alpha,a+1} : \alpha \in \Phi_{k+1-b}\}.\]
%\langle \alpha, \Lambda \rangle \leq a \text{ for } \alpha \in \Phi_b \text{ and } \langle \alpha, \Lambda \rangle \leq a+1 for \alpha \in \Phi_{b-k+1}\}.\]
\end{defn}

The following proposition is related to Lemma 2.2 in \cite{Fan} and Theorem 5.7 in \cite{Som}.

\begin{prop}
	The alcoves contained in $D_m^{k}$ are in bijection with the alcoves contained in $m\Anull.$
\end{prop}

\begin{proof}
We show that the reflections in the hyperplanes in $H_m$ generate a copy of $\widetilde{\A}_k$.  
This implies that $D_m^{k}$ is a fundamental region for this copy of $\widetilde{\A}_k$, so that $D_m^{k}$ 
is \emph{similar} to a dilation of $\Anull$.  To identify the size of the dilation, 
we choose a hyperplane defining a facet and compute the hyperplane that passes through 
the vertex of the region opposite to it.
%We have that $\alpha$

Using one-line notation, we therefore have generators 
\[((i,i+b+a(k+1)))=((i-a(k+1),i+b)) \text{ for } 1\leq i \leq k-b+1,\] 
corresponding to the reflections in the hyperplanes $H_{\alpha_i+ \alpha_{i+1} + \cdots \alpha_{i+b-1}, -a}$ and
\[((i,i+k+1-b-(a+1)(k+1)))=((i,i-b-a(k+1))) \text{ for } 1 \leq i \leq b,\] 
corresponding to the reflections in the hyperplanes $H_{\alpha_i+ \alpha_{i+1} + \cdots \alpha_{i+k-b}, a+1}$.
We may rewrite these in a unified manner as the reflections 
\[T_m := \{((j,j+b+a(k+1))) : 1 \leq j \leq k+1\},\] 
so that the first $k+1-b$ of these reflections correspond to the first set of reflections and the last $b$ correspond to 
second set.  We shall trace through the affine Dynkin diagram generated by these reflections.

Let us suppose that $b< k/2$; the argument for $b>k/2$ is analogous.  Define a bijection 
$f:T_m \to \mathbb{Z}_{k+1}$ by \[f((j,j+b+a(k+1))):=j.\]  Since $b$ and $k+1$ are relatively prime, 
the sequence $\{f^{-1}(sb)\}_{s=0}^{k+1}$ contains each reflection exactly once.  Moreover, 
one can check that 
\[f^{-1}(sb) = ((sb, (s+1)b +(k+1)a))\text{ and } f^{-1}((s+1)b) = (((s+1)b, (s+2)b + (k+1)a))\] 
satisfy $\left(f^{-1}(sb) f^{-1}((s+1)b) \right)^3=1$,
and that non-adjacent reflections commute.  We
conclude that reflections in the hyperplanes of $H_m$ generate a copy of $\widetilde{\A}_k$.

We now find the exact size of the dilation by computing $m'$ so that the hyperplane 
$H_{\alpha_1+\cdots+\alpha_{b},m'}$ contains the weight obtained by intersecting all hyperplanes except 
$H_{\alpha_1+\cdots+\alpha_{b},-a}$ (the calculation below may be extended to the omission of any 
hyperplane, but we will not use this here).  If this weight has the form 
$\eta=\sum_{i=1}^k \eta_k \Lambda_k$, we therefore need to compute the sum of the first 
$b$ coordinates $m'=\sum_{i=1}^b \eta_i$.

We have the following equations specified by the hyperplanes:

\begin{align*}
E_1 =\eta_1+\eta_2+\cdots+\eta_{k-b+1} &= a+1 \\
E_2 =\eta_2+\eta_3+\cdots+\eta_{k-b+2} &= a+1 \\
\cdots \\
E_b =\eta_{b}+\eta_{b+1}+\cdots+\eta_{k} &= a+1 \\
E_{b+1} =\eta_2+\eta_3+\cdots+\eta_{b+1} &= -a \\
E_{b+2} =\eta_3+\eta_4+\cdots+\eta_{b+2} &= -a \\
\cdots \\
E_{k} =\eta_{k+1-b}+\eta_{k+2-b}+\cdots+\eta_{k} &= -a.
\end{align*}

Then we compute that $\sum_{i=1}^b E_i - \sum_{i=b+1}^k E_i = \sum_{i=1}^b \eta_i = (a+1)b+a(k-b)=b+ka=m'$.  
In particular, the hyperplane that passes through $\eta$ is $H_{\alpha_1+\cdots+\alpha_{b},b+ka}$, 
so that the distance between it and the facet opposite to it is $m=(k+1)a+b$, as desired.  Therefore, 
$D_m^{k}$ is similar to $m\Anull$.
\end{proof}

%	Suppose the reflection in the hyperplane $H_{\alpha,0}$ has reduced word $t_{\alpha,0}$, where $t_{(i,j+1),0}:=s_i s_{i+1} \cdots s_{j-1} s_j s_{j-1} \cdots s_{i+1} s_i$.  A reflection in the hyperplane $H_{\alpha,1}$ has reduced word $t_{\alpha,1}:=s_0 t_{t_{\alpha_0,0} (\alpha),0} s_0$ for $\alpha \neq \alpha_0$, while a reflection in $H_{\alpha_0,1}$ has reduced word $s_0$.  A reflection in $H_{\alpha,p}$ has reduced word $t_{\alpha,p} := t_{\alpha,p-1}t_{\alpha,p-2}t_{\alpha,p-1}.$

\begin{lemma}
The reflection $((j,j+b+a(k+1)))$ has reduced word \[\phi^{j-1}((s_1 \cdots s_k s_0)^a (s_1 s_2 \cdots s_b \cdots s_2 s_1) (s_1 \cdots s_k s_0)^{-a})\] for $1 \leq j \leq k+1$.
 %T_m := \{((j,j+b+a(k+1))) : 1 \leq j \leq k+1\} have reduced words
\label{lem:rwforhyperplanes}
\end{lemma}

\begin{proof}
For simplicity, let us consider the case $a=0$.  It is clear by the action of $s_i$ on one-line notation 
that the reflection $((1,1+b))$ has reduced word $s_1 s_2 \cdots s_b \cdots s_2 s_1$.  The map 
$\phi$ shifts the simple reflections by one, so that 
$\phi^{j-1}(s_1 s_2 \cdots s_b \cdots s_2 s_1)$ acts as the reflection $((j,j+b))$.  
The result for $a>0$ follows identically.
\end{proof}

\begin{cor}
	$D_m^{k}$ is invariant under $\phi$ and therefore defines a cyclic action on the alcoves in $D_m^k$.
\end{cor}

\begin{proof}
We note that $\phi$ is a continuous map since for vectors $\eta = \sum_{i=1}^k \eta_i \wt_i$ and
$\tau = \sum_{i=1}^k \tau_i \wt_i$ and for real values of $\epsilon$, we have
\[ \phi( \eta + \epsilon \tau ) - \phi( \eta ) %=
%\sum_{i=1}^k (\eta_i + \epsilon \tau_i) \wt_{i+1} + (1-\sum_{i=1}^k (\eta_i - \epsilon \tau_i)) \Lambda_1
%- \sum_{i=1}^k \eta_i \wt_{i+1} - (1-\sum_{i=1}^k \eta_i) \Lambda_1
= \epsilon \sum_{i=1}^k \tau_i \wt_{i+1} - \epsilon \sum_{i=1}^k \tau_i \Lambda_1~.
\]
Notice also that $\phi$ preserves the set of reflections $T_m$ and hence also preserves the
set of hyperplanes $H_m$.  We conclude that since $\phi$ is continuous and preserves the
bounding hyperplanes of $D_m^k$ that it maps $D_m^k$ to itself.

\end{proof}

\begin{figure}[ht]
\begin{center}
\includegraphics[scale=.9]{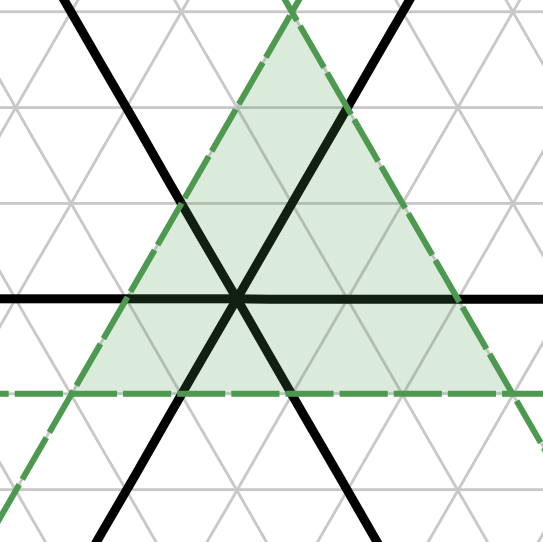}
\end{center}
\caption{The shaded region is $D_4^2$ which is bounded by the hyperplanes $\{ H_{\alpha_1,-1}, H_{\alpha_2,-1},
H_{\alpha_1+\alpha_2,2}\}$.  The alcoves furtherest away from $\Anull$
are indexed by the affine symmetric group
elements $s_0 s_2 s_0 s_1$, $s_1 s_0 s_1 s_2$ and $s_2 s_1 s_2 s_0$. }%and 
%the corresponding 3-bounded partitions are obtained by deleting shaded cells and left justifying the partition.
\label{eg:D24}
\end{figure}

\begin{eg}  The region $D_3^3$ looks exactly like the shape in Figure \ref{fig:3dkeq} and is cut out
by the bounding hyperplanes 
$\{ H_{\alpha_1+\alpha_2+\alpha_3,0}, H_{\alpha_1,1}, H_{\alpha_2,1}, H_{\alpha_3,1} \}$.
The four longest elements of $D_3^3$ in $\widetilde{\A}_3$ are 
$s_1 s_0 s_2 s_3$, $s_2 s_1 s_3 s_0$, $s_3 s_2 s_0 s_1$, and $s_0 s_3 s_1 s_2$.  
They form a single orbit of size four under $\phi$.  By contrast, the elements 
$s_0s_2$ and $s_1s_3$ form an orbit of size two under $\phi$, and the identity is in an orbit by itself.
\end{eg}

\subsection{Enumeration of orbits when $m=2$}

\begin{defn}[Reiner, Stanton, White~\cite{reiner2004cyclic}]
Let $X$ be a finite set, $C$ a cyclic group of order $k+1$ acting on $X$, and 
$X(q)$ a generating function for $X$.  Then the triple $(X, C, X(q))$ 
\deft{exhibits the cyclic sieving phenomenon (CSP)} if for $c \in C,$ 
\[X(\omega(c)) = |\{ x \in X : c(x) = x \}|,\] 
where $\omega: C \to \mathbb{C}$ is an isomorphism of $C$ with the $(k+1)$st roots of unity.
\end{defn}

In other words, the triple $(X, C, X(q))$ exhibits the CSP if the orbit structure of $X$ 
under a cyclic action may be computed by evaluating the polynomial $X(q)$ at a root of unity.

We recall here the description of R.~Suter's cyclic action on $Y^k$: given a partition 
$\lambda = (\lambda_1, \lambda_2, \ldots, \lambda_n)$ such that $\lambda_1 + n-1 \leq k$, 
$\Ct_2^k$ acts on the Ferrers diagram $\lambda$ by removing its bottom row $\lambda_1$ and 
prepending a column of height $k-\lambda_1$ on the left.

V.~Reiner conjectured that $(Y^k, \langle \Ct_2^k \rangle, \prod_{i=1}^{n-1} (1+q^i))$ 
exhibits the cyclic sieving phenomenon (see~\cite{TW} for a full account).  
D.~Stanton stated the more refined conjecture that there was an equivariant bijection between 
$Y^k$ under R.~Suter's cyclic action and binary words of length $n$ with odd sum.  
The following proof of this conjecture originally appeared in~\cite{nathanmasters}, 
where it was called a \deft{bijaction} (a bijection induced by an action).  
A generalization of this result for any $m$ is given in~\cite{TW}.
%M.~Visontai and N.~Williams generalized it the case $m=3$ in~\cite{mirko} and to general $m$ in~\cite{TW}.

\begin{thm}
There is an equivariant bijection between $Y^k$ under $\Ct_2^k$ and binary words of length $k+1$ with odd sum under rotation. 
\label{thm:equivy}
\end{thm}

%Given a partition $\mu = (\mu_1, \mu_2, \ldots, \mu_k)$ such that $\mu_1 + k-1 \leq n-1$, form its Ferrers diagram.  We act on $\mu$ with the generator $c_Y$ by removing the top row of the diagram and prepending a column of height $n-\mu_1-1$ on the left.

%\begin{theorem} [R.\ Suter]
%\label{thm:suter}
%The cyclic group $C_{n}$ acts on $Y_n$ by $c_Y$.
%\end{theorem}

\begin{proof}
Let $X^k$ be the set of binary words of length $k$.  To define the bijection from $Y^k$ to 
binary words of odd sum, we first recall a bijection between the set of partitions $Y^k$ and $X^k$.  
Beginning at the bottom right of the Ferrers diagram of a partition $\lambda$, trace along the 
boundary of $\lambda$ to the top left of the diagram, transcribing steps left as $1$ and steps up as $0$.  
Omit the first step, which must be up, and finally pad the transcribed word on the right with $0$s 
to get a word $x(\lambda)$ of length $k$.  We call $x(\lambda)$ the \deft{boundary word} for the 
partition $\lambda$.  For example, when $k=6$, $x(221) = 010100$.%left(\young{&\cr&\cr\cr}\right)

We next define the cyclic action $\Cb_2^k$ directly on binary words of length $k$.  Map a binary word 
$x=\underset{n}{\underbrace{1\ldots1}}0y$, where $y$ represents the last $k-n-1$ letters of $x$, 
to the word $y1 \underset{n}{\underbrace{0\ldots0}}$, where $n$ may be zero.  Then this is exactly 
the action of $\Ct_2^k$, since removing the top row is equivalent to removing an initial string of 
$11 \ldots 10$, and adding in a new first column corresponds to appending the string $100 \ldots 0$.

For each binary boundary word, we now perform this action $k+1$ times, appending strings that correspond 
to the first columns without removing the initial string that corresponds to the top row.  We may restrict 
the length of the resulting word to be $2(k+1)$, since the action is of order $k+1$.  It is clear that 
these extended binary words are of the form $\bar{x} = x 1 (1-x)0$, where $x$ is our original binary 
boundary word and addition is done modulo $2$; we write \[W^k = \{x 1 (1-x)0 : x \in X^k\}.\]  
Then $\Cb_2^k$ acts on these extended binary words by rotating a word left until just after the leftmost $0$.  
%Since we will be using powers of the action, we will write $c$ for the action of $C_2^k$ on $X^k$ or $W^k$.

We have therefore established equivariant bijections between $Y^k$ under $\Ct_2^k$, $X^k$ under $\Cb_2^k$, and 
$W^k$ under $\Ch_2^k$.  We will now give an equivariant bijection between $X^k$ under the map $\Cb_2^k$
and binary words of  length $k+1$ with odd sum under rotation.

%For the above cycle, this gives $$1100100110 \to 0100110110 \to 1001101100 \to 0110110010 \to 1101100100 \to 1100100110.$$

Let $s(x)$ be the first letter of the word $x$.  We map $x \in X^k$ to the word $w(x)$ defined by by 
\[w(x):=s(x) s((\Cb_2^k)^{1}(x)) s((\Cb_2^k)^{2}(x)) \cdots s((\Cb_2^k)^{k} (x)).\]

The claim is that this is the desired bijection.  By construction, the map $x \mapsto w(x)$ is equivariant
where the action of $\Ch_2^k$ acts on elements of $W^k$ by taking the first letter of $w(x)$ and moving it
to the end.  

Furthermore, $w(x)$ has sum $1 \mod 2$: by the definition of $\Ch_2^k$ on $W^k$, $w(x)$ is the word 
containing those numbers that occur after a $0$ in $\bar{x}$.  If 
\[\bar{x}=\bar{x}_1\bar{x}_2 \cdots \bar{x}_{2(k+1)} = x_1 x_2 \cdots x_k 1 (1-x_1)(1-x_2)\cdots (1-x_{k}) 0,\] then 
$x_{i+1}$ occurs in the sum of the letters of $w(x)$ iff $x_i=0$, while $(1-x_{i+1})$ occurs in the sum 
iff $x_i = 1$.  Therefore each letter $x_{i+1}$ will be counted as $x_{i+1}-x_i$ and so we can express 
the sum of the letters as $x_1 + \sum_{i=1}^{k-1} (x_{i+1} - x_{i}) + (1-x_k) = 1 \mod 2$.  

Therefore, the image of $W^k$ under the map $w$ is indeed a subset of words of length $k+1$ that sum to 
$1 \mod 2$.  It remains to show that $w$ is a bijection, which we shall do by constructing its inverse.  
If $w=w_1 w_2 \cdots w_{k+1}$ is a word that sums to $1 \mod 2$ with $2n+1$ ones, we split $w$ into two pieces, 
$w^{(0)}=w_1 w_2 \cdots w_\ell$ and $w^{(1)}=w_{\ell+1} w_{\ell+2} \cdots w_{k+1}$, where $\ell$ is the 
position of the $(n+1)$st $1$ in $w$ (from the left).  We can now recover $x=x_1 x_2 \cdots x_k$ as follows.  
Let $\sigma:=0$.  For $i=1,2,\ldots,k+1$, let $t$ be the leftmost unread letter in $w^{(\sigma)}$, update 
$\sigma=\sigma+t$, and record $x_i:=\sigma$.  By definition of $\ell$, this procedure terminates by visiting 
all the letters of $w$, and it is not hard to check that this procedure recovers the word $x1$.
\end{proof}
%the word $w(x)$ to which $x$ bijacts is of length $n$, since there are $n$ $0$s in the extended word for $x$.  
%Furthermore, $w(x)$ has sum $m-1 \mod m$: $w(x)$ is the word containing those numbers that occur after $0$ in $\bar{x}$, so that the letter $x_{i+1}$ will be counted in the sum as $x_{i+1}-x_i$.  Therefore, we can express the sum of the letters as $w_1 + \sum_{i=1}^{n-1} (x_{i+1} - x_{i}) = m-1 \mod m$.  
\begin{eg}
For example, an orbit in $Y^4$ is \[(2) \to (1,1) \to (2,1,1) \to (2,2) \to (3,1),\] with corresponding orbit in 
$X^4$ \[1100 \to 0100 \to 1001 \to 0110 \to 1101.\]  Taking the first letters of this orbit, starting with $1100$, 
gives $w(1100)=10101$ (a word of length $k+1=5$ that sums to $1 \mod 2$). Its orbit under $\Ch_2^4$ is
\[ 10101 \to 01011 \to 10110 \to 01101 \to 11010 \]
  To reverse the bijection, the word 
$w=10101$ is split into $w^{(0)}=101$ and $w^{(1)}=01$.  Then

\begin{itemize}
\item $(i=1, t=1, \sigma=1, \bar{1}01|01): x_1=1,$
\item $(i=2, t=0, \sigma=1,\bar{1}01|\bar{0}1): x_1x_2=11,$
\item $(i=3, t=1, \sigma=0, \bar{1}01|\bar{0}\bar{1}): x_1x_2x_3=110,$
\item $(i=4, t=0, \sigma=0, \bar{1}\bar{0}1|\bar{0}\bar{1}): x_1x_2x_3x_4=1100,$
\item $(i=5, t=1, \sigma=1, \bar{1}\bar{0}\bar{1}|\bar{0}\bar{1}): x_1x_2x_3x_4x_5=11001.$
\end{itemize}
\end{eg}

\begin{cor}
$(2\Anull, \langle C_2^k \rangle, \prod_{i=1}^{n-1} (1+q^i))$ exhibits the CSP.
\end{cor}

\begin{proof}
Theorem~\ref{thm:equivy} gives an equivariant bijection between between $Y^k$ under $\Ct_2^k$ and 
binary words of length $n$ with odd sum under rotation.  Recall that $[m]_q:=\frac{1-q^{m}}{1-q},$ $[m]_q!:=\prod_{i=1}^m [i]_q,$ and $\binom{m}{k}:=\frac{[m]_q!}{[m-k]_q![k]_q!}.$  By Theorem 1.1 of \cite{reiner2004cyclic}, we know that
binary words of length $n$ with odd sum under rotation exhibit the cyclic sieving property with
generating function
$$S(q) = \sum_{k~odd} {{n}\brack{k}}_q.$$

It remains to show that $S(q) \equiv \prod_{i=1}^{n-1} (1+q^i) \mod q^n-1$.
This fact follows from an argument provided to us by D. Stanton \cite{stanton}.
Let $\zeta$ be a $d^{th}$ root of unity where $d$ divides $n$.  Then 
$${m \brack k}_{q=\zeta} = 
\begin{cases}{\lfloor m/d\rfloor \choose k/d}&\hbox{ if $d|k$}\\
0 & \hbox{ otherwise}
\end{cases}.
$$

By the $q$-binomial theorem,
\begin{align}
\prod_{i=1}^{n-1} (1+q^i) &= \sum_{j=0}^{n-1} q^{\frac{j(j+1)}{2}} {n-1 \brack j}_q.
\end{align}

When $q=\zeta$, we have
\begin{align}
\prod_{i=1}^{n-1} (1+\zeta^i) &= \sum_{j=0}^{n-1} \zeta^{\frac{j(j+1)}{2}} {n-1 \brack j}_{q=\zeta}
= \sum_{j: d|j} (-1)^{\frac{j}{d}(j+1)}{n/d-1 \choose j/d }
= \begin{cases}
0&\hbox{ if $d$ is even}\\
2^{n/d-1}&\hbox{ if $d$ is odd}\\
\end{cases},
\end{align}
while we also have that
\begin{equation}
\sum_{k=0}^{n/2} {n \brack 2k+1}_q
= \sum_{k: d|(2k+1)} {n/d \choose (2k+1)/d}
= \begin{cases}
0&\hbox{ if $d$ is even}\\
2^{n/d-1}&\hbox{ if $d$ is odd}\\
\end{cases}~.
\end{equation}
\end{proof}

\section{Acknowledgements}
The authors would like to thank Hugh Thomas for numerous conversations and input. 
The authors are also grateful to Drew Armstrong for pointing out the paper and 
corresponding results  of Sommers \cite{Som}. The authors would also like to 
thank the reviewers for their helpful comments.

The authors are also grateful to the anonymous reviewer that offered
suggestions to improve an earlier version of this paper.

\section{Appendix : Example for $m=4$ and $k=6$}

We provide one large example of the graph structure of
the poset of partitions in $Y_m^k$ when $m=4$ and $k=6$
in Figure \ref{fig:largeexamp}.  By Proposition~\ref{prop:volume} we know that this
graph has $4^6 = 4096$ vertices and the program generated $10752 = 7 \cdot 3 \cdot 2^9$ edges. 
  
An edge occurs between two partitions
if their images under the map
$\mathfrak{r} : \mathcal{P}^{(k)} \rightarrow \widetilde{\A}_k$
are comparable in weak order.  There is also a combinatorial method for checking
this condition.  There is an edge between two partitions $\bndd$ and $\la$
if they differ by a single cell and $\bndd^{\omega_k}$ and $\la^{\omega_k}$
also differ by a single cell (recall that the operation $\bndd^{\omega_k}$ was defined
in Section \ref{sec:kbounded}).

The picture in Figure \ref{fig:largeexamp} was created using programs in Sage
for manipulating partitions and cores \cite{sage, sage-co} to generate a file
that describes the graph structure.  The graph was then displayed using the program
{\tt neato} in the Graphviz software package \cite{graphviz} to generate the image.

\begin{figure}[ht] \label{fig:largeexamp}
\begin{center}
\includegraphics[width=6in]{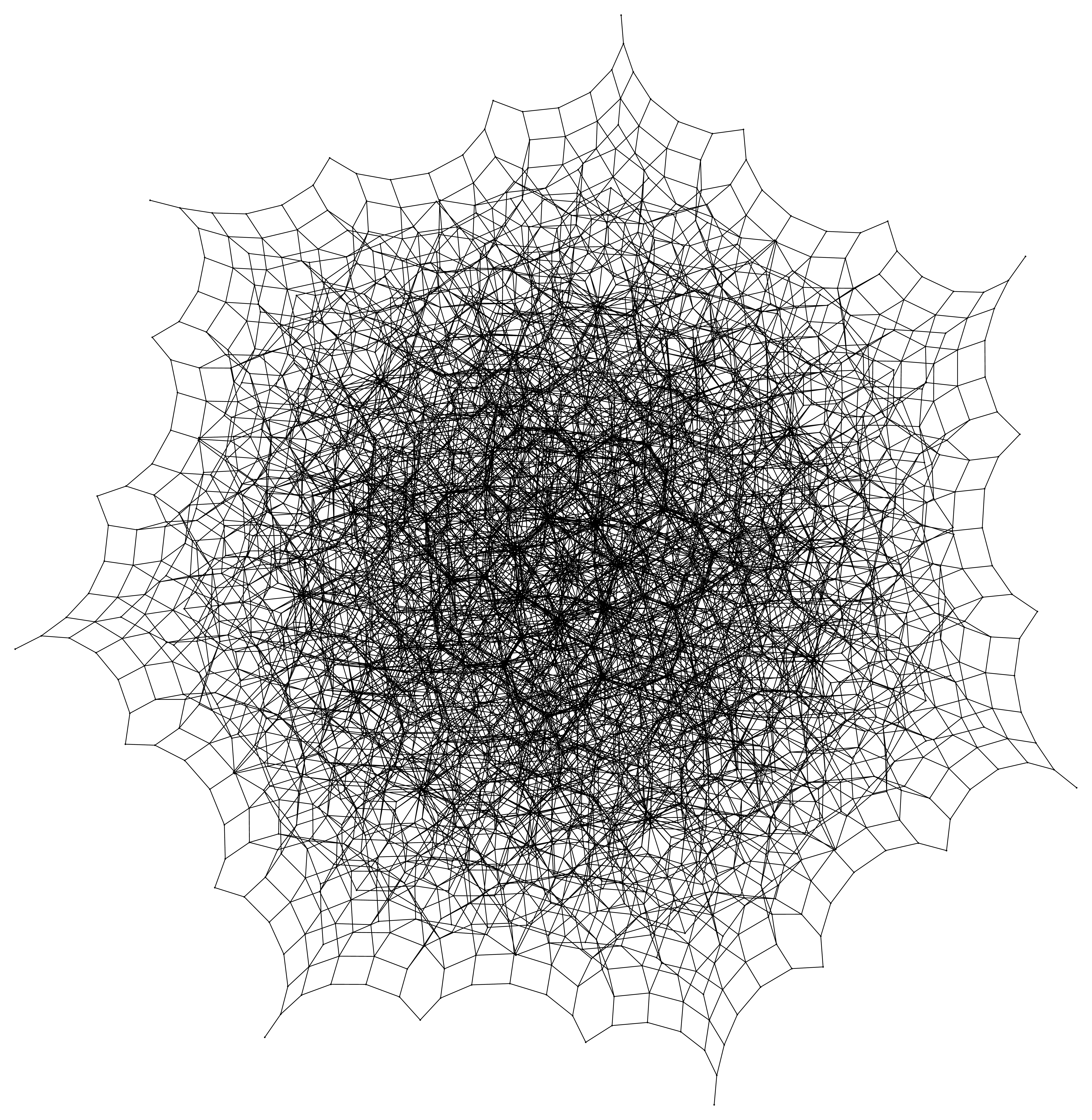}
\caption{Poset of $Y_4^6$ consisting of partitions which are contained in unions of at most
four rectangles with hook-length $6$.  This picture was created using the programs Sage \cite{sage, sage-co}
and Graphviz \cite{graphviz}.}
\end{center}
\end{figure}

\end{document}